\documentclass[10pt,reqno]{amsart}

\usepackage{amsaddr}
\usepackage{etoolbox}
\patchcmd{\section}{\scshape}{\bfseries\scshape}{}{}
\makeatletter
\renewcommand{\@secnumfont}{\bfseries}
\makeatother

\input{our_macros.sty}

\usepackage{tikz}
\usetikzlibrary{patterns,decorations,arrows,shapes.geometric,arrows,mindmap,backgrounds,positioning,positioning,fit,backgrounds,positioning,arrows,matrix,calc}
\usetikzlibrary{shapes,backgrounds,mindmap,shadows,shapes.geometric,topaths,snakes,arrows,pgfplots.groupplots,fadings,calc,decorations.markings,positioning,chains,fit}

\title{Statistical limits of correlation detection in trees}
\author{Luca Ganassali}
\address{\vspace{-0.2cm}{\footnotesize Inria, DI/ENS, PSL Research University, Paris, France}\vspace{-0.5cm}}
\author{Laurent Massoulié}
\address{\vspace{-0.2cm}{\footnotesize MSR-Inria Joint Centre, Inria, DI/ENS, PSL Research University, Paris, France}\vspace{-0.5cm}}
\author{Guilhem Semerjian}
\address{\vspace{-0.2cm}{\footnotesize Laboratoire de Physique de l’École normale supérieure, ENS, Université PSL, CNRS, Sorbonne Université, Université Paris Cité, F-75005 Paris, France}\vspace{-0.5cm}}
\date{\today}

\begin{document}

\begin{abstract}
	In this paper we address the problem of testing whether two observed trees $(t,t')$ are sampled either independently  or from a joint distribution  under which they are correlated. 
	This problem, which we refer to as \emph{correlation detection in trees}, plays a key role in the study of graph alignment for two correlated random graphs. Motivated by graph alignment, we investigate the conditions of existence of one-sided tests, i.e. tests which have vanishing type I error and non-vanishing power in the limit of large tree depth. 
	
	For the correlated Galton-Watson model with Poisson offspring of mean $\lambda>0$ and correlation parameter $s \in (0,1)$, we identify a phase transition in the limit of large degrees at $s = \sqrt{\alpha}$, where $\alpha \sim 0.3383$ is Otter's constant. Namely, we prove that no such test exists for $s \leq \sqrt{\alpha}$, and that such a test exists whenever $s > \sqrt{\alpha}$, for $\lambda$ large enough.
	
	This result sheds new light on the graph alignment problem in the sparse regime (with $O(1)$ average node degrees) and on the performance of the \texttt{MPAlign} method studied in \cite{GMLTrees2021journal,piccioli2021aligning}, proving in particular the conjecture of \cite{piccioli2021aligning} that \texttt{MPAlign} succeeds in the partial recovery task for correlation parameter $s>\sqrt{\alpha}$ provided the average node degree $\lambda$ is large enough. 

 {As a byproduct, we identify a new family of orthogonal polynomials for the Poisson-Galton-Watson measure which enjoy remarkable properties. These polynomials may be of independent interest for a variety of problems involving graphs, trees or branching processes, beyond the scope of graph alignment.}
\end{abstract}

\maketitle

\section*{Introduction}

Upon the observation of two unlabeled rooted trees $t$ and $t'$ of depth at most $d$, how well can the statistician tell whether the trees are correlated or independent? This fundamental statistical task, namely correlation detection in trees, has been introduced in \cite{Ganassali20a} and further studied in \cite{GMLTrees2021journal} and \cite{piccioli2021aligning}. The case of \ER graphs has been investigated \cite{Barak2019}: information-theoretic limits are given in \cite{Wu20,Ding22b}, and \cite{Mao21_counting} proposes an algorithm based on counting trees which succeeds in polynomial time in some regime of parameters. 

Correlation detection in trees can be defined per se and studied as such; the related question of finding proximity measures between unlabeled trees has been addressed in fields as diverse as e.g. combinatorics \cite{Micheli2006}, theoretical computer science \cite{WuHuang2010}, and phylogenetics \cite{Choi2009}.

Recent work \cite{Ganassali20a,GMLTrees2021journal,piccioli2021aligning} also shows that this problem arises naturally in the study of \emph{graph alignment} \cite{Cullina2017,Cullina18,Wu2021SettlingTS, fan2019ERC} in the following manner. Consider the so-called correlated \ER graph model which consists of two graphs $G,G'$ on node set $[n]$ where for each pair of  nodes $i<j\in [n]$, the indicators of $(i,j)$ edge presence in the two graphs are Bernoulli random variables with parameter $q \in (0,1)$ and correlation $s \in (0,1)$. Let then graph $H$ be obtained from $G'$ by relabeling its nodes according to some uniformly random permutation $\pi^\star$. 

In this setup, graph alignment consists in recovering $\pi^\star$ from the observation of $(G,H)$. A potential approach aims to determine whether node $u$ of $G$ is matched to node $u'$ of $H$, namely whether $u' = \pi^\star(u)$, based on the local structures of the graphs $G$, $H$ around these candidate vertices $u$, $u'$. 
With this procedure,  considered in \cite{Ganassali20a,GMLTrees2021journal,piccioli2021aligning}, one forms the estimator $\hat\pi$ 
such that $\hat\pi(u)=u'$ if and only if the local structure of graph $G$ in the neighborhood of node $u$ is somehow \emph{close} (or more precisely, \emph{correlated}) to the local structure of graph $H$ in the neighborhood of node $u'$. 

In the sparse regime where $q= \lambda/n$ with $\lambda = \Theta(1)$, it is well known \cite{Benjamini2011} that the neighborhoods up to fixed distance $d$ of node $u$ (resp. $u'$) in $G$ (resp. $G'$), are both asymptotically as $n\to\infty$ distributed as Galton-Watson branching trees. More specifically, if $u' = \pi^\star(u)$, then the two neighborhoods are asymptotically jointly distributed as correlated Galton-Watson branching trees -- a distribution denoted $\dPls_{d}$. On the other hand, for pairs of nodes $(u,u')$ taken at random in $[n]$, the two neighborhoods are asymptotically independent Galton-Watson branching trees -- with distribution denoted $\dPl_{d}$. Hence detecting correlation in random trees appears naturally as a subroutine to solve the recovery task in the random graph alignment problem. 

In the sparse regime, random graph alignment is particularly challenging. Indeed, only a partial fraction of the nodes can be recovered by any estimator, as shown by \cite{Cullina2017}. More precisely, the best fraction one can hope for is upper-bounded by the typical proportion of nodes in the giant component of the underlying intersection graph (i.e. the graph containing the edges in both $G$ and $G'$), implying that only a vanishing fraction can be aligned when the average degree $\lambda s$ of the intersection graph verifies $\lambda s\leq 1$,  see \cite{ganassali2021impossibility}. Information-theoretic bounds for partial recovery have been progressively improved \cite{Hall-LM}, \cite{Wu2021SettlingTS}, culminating with the result of \cite{Ding22} which shows that a non-vanishing fraction of nodes can be aligned provided  $\lambda s > 1$. Combined with the aforementioned result of \cite{ganassali2021impossibility}, it implies that $\lambda s=1$ is the threshold for information-theoretic partial alignment.

Other recent works have provided algorithms for deciding whether the two local structures up to distance $d$ are correlated, namely if the pair $(u',u)$ corresponds to the same node in the underlying intersection graph $G \land G'$. As further discussed in Remark~\ref{remark:types_de_tests}, the right notion of test in the present context is that of \emph{one-sided tests}, namely tests for deciding between the null hypothesis (trees uncorrelated) and the alternative hypothesis (trees correlated) with vanishing type I error and non-vanishing power as the tree depth $d$ increases.

A first approach proposed in \cite{Ganassali20a} consists in a test based on a measure of similarity between two trees: the \emph{tree matching weight}, defined as the maximal size of a common subtree, measured by its number of leaves. \cite{Ganassali20a} shows that a one-sided test can be constructed from the tree matching weight statistic for the parameter range $\{(\lambda,s):\lambda\in(1,\lambda_0],\; s\in(s^*(\lambda),1]\}$, where $\lambda_0>1$ and $s^*(\lambda)<1$.

The work \cite{GMLTrees2021journal} studies the optimal test based on the likelihood ratio and characterizes ranges of parameters where it succeeds or fails at being one-sided. 
{Then, they} propose a message-passing algorithm for graph alignment, \texttt{MPAlign}, that is naturally inspired by the related problem on trees. {The study} goes on to show that, in some parameter ranges, \texttt{MPAlign} provably returns a partial matching containing a non-vanishing fraction of correctly aligned nodes and a vanishing fraction of misaligned vertices, i.e. achieves one-sided partial recovery. The connection between the tree and the graph problems is formalized by the following.

\begin{theorem}[\cite{GMLTrees2021journal}, Theorem 2]\label{theorem:MPalign}
	For a given $(\lambda,s)$, if there exists a one-sided test for correlation detection in the tree problem, then one-sided partial alignment in the correlated \ER model is achieved in polynomial time by the algorithm \texttt{MPAlign}.
\end{theorem} 

In parallel, \cite{piccioli2021aligning} studied this same algorithm, giving also numerical results and insights for extensions to more general models. For the correlated Galton-Watson models, the authors conjectured -- among other things -- that $s = \sqrt{\alpha}$ is the sharp threshold for one-sided detection asymptotically in the limit $\lambda \to \infty$, where $\alpha \sim 0.3383$ is the Otter's constant defined below in Proposition \ref{prop:otter}. 

The goal of this study is to prove the previous conjecture, hence improving the understanding of correlation detection in trees. 
Most relevant to this conjecture and to the present paper is the recent work by Mao, Wu, Xu and Yu \cite{Mao21_counting} previously mentioned, which studied the correlation detection problem in \ER graphs. The authors propose an algorithm based on counting (signed) trees, which can provably distinguish graph correlation efficiently as soon as $s > \sqrt{\alpha}$, for any average degree. 

The results presented here are different for several reasons: first, we study the detection problem on trees instead of \ER graphs. Also, as we consider an optimal test, we are able to show that $s \leq \sqrt{\alpha}$ implies impossibility of one-sided detection, so that one-sided detection exhibits a sharp threshold at $s = \sqrt{\alpha}$, asymptotically in $\lambda$, see Figure \ref{fig:phase_diagram_new}. 
{The combination of our positive result on the existence of one sided tests for the detection problem on trees for $s > \sqrt{\alpha}$ and $\lambda$ large enough with Theorem~\ref{theorem:MPalign} implies the existence of a polynomial-time algorithm for partial recovery in the correlated random graph problem in this regime of parameters; the same conclusion was reached in an independent and almost simultaneous work~\cite{Mao23chandeliers}, their algorithm being based on counting trees of a certain kind called chandeliers. Otter's threshold $\sqrt{\alpha}$ in the correlation parameter also plays a crucial role in~\cite{Mao23chandeliers}, which in addition handles the (almost) exact recovery task when the average degree of the graphs diverge with the size.}

We believe {first} that this study paves the way for further work in this field such as generalizations to other graph models
and design of more efficient algorithms for tree correlation detection or graph alignment. 
{Second, the scope of this article goes beyond this field of research. Indeed, the core result of the paper is the decomposition of the likelihood ratio (see Theorem \ref{theorem:eigendecomposition}). In this decomposition, we identify a new family of orthogonal polynomials for the Poisson-Galton-Watson measure which enjoy remarkable properties (see Remark \ref{remark:charlier}). These polynomials may be of independent interest for a much broader variety of problems involving graphs, trees or branching processes.}

\section{Problem statement, main results}
\subsection{Definitions, notations}
Throughout the paper, $\dN$ (resp. $\dN_+$) denotes the set of non-negative (resp. positive) integers, and $\Poi(\mu)$ denotes the Poisson distribution of parameter $\mu >0$, namely for $k \in \dN$
$$\Poi(\mu)(k) = e^{-\mu}\frac{\mu^k}{k!} \, .$$  

If $P$ and $Q$ are two probability measures on the same measurable space $\cX$ such that $P$ is absolutely continuous with respect to $Q$, the Kullback-Leibler divergence (or relative entropy) from $Q$ to $P$ is defined as follows:

$$ \KL(P \Vert Q) := \int_{\cX} \log \frac{dP}{dQ}(x) P(dx) = \dE_P \left[ \log \frac{dP}{dQ}(X) \right]\, .$$

{A central object of study in this paper will be unlabeled rooted trees; these can be defined as equivalence classes of labeled rooted trees, quotiented by parenthood preserving isomorphisms, and this is probably the most usual definition. For our purposes a slightly different, yet equivalent, point of view will turn out to be more useful, in order to exploit the recursive nature of tree structures. Indeed, a labeled rooted tree $T$ can be naturally described by the degree of its root, say $\ell$, and the list $T_1,\dots,T_\ell$ of the subtrees rooted at the offsprings of the root. Such a decomposition is ill-behaved under the unlabelling transformation, since it requires to specify which among the $T_1,\dots,T_\ell$ correspond to isomorphic subtrees. This led us to the following formalization, that encompasses simultaneously the recursive nature of trees and the symmetrization of their unlabeled versions (see Fig.~\ref{fig_def_trees} for an illustration):}

\begin{figure}
    \centering
    \begin{tikzpicture}
        \fill[black] (0,0) circle (2pt) node [above] {$\varnothing$};
        \fill[black] (-1,-1) circle (2pt) node [left] {$1$};
        \fill[black] (-1,-2) circle (2pt) node [left] {$11$};
        \fill[black] (0,-1) circle (2pt) node [left] {$2$};
        \fill[black] (1,-1) circle (2pt) node [left] {$3$};
        \fill[black] (1,-2) circle (2pt) node [left] {$31$};
        \draw (-1,-2) -- (-1,-1) -- (0,0) -- (1,-1) -- (1,-2);
        \draw (0,0) -- (0,-1);
        \draw (-1.4,-2.2) -- (-1.6,-2.2) -- (-1.6,0.5) -- (-1.4,0.5);
        \draw (1.3,-2.2) -- (1.5,-2.2) -- (1.5,0.5) -- (1.3,0.5);
        \draw (3.5,-1) node {$= \ \ \{N\ = 1 \ , \ \ N\ = 2 \}$};
    \fill[black] (2.85,-1.1) circle (2pt);
    \fill[black] (4.4,-1.1) circle (2pt);
    \fill[black] (4.4,-1.35) circle (2pt);
    \draw (4.4,-1.35) -- (4.4,-1.1);
    \end{tikzpicture}
    \caption{An illustration of Definition~\ref{def:rooted_unlabeled_trees} with an unlabeled rooted tree of depth $d=2$, represented as the equivalence class of a canonically labeled tree, and in terms of the number of copies of trees of depth $1$ rooted at the offsprings of the root.}
    \label{fig_def_trees}
\end{figure}

\begin{definition}[Finite rooted unlabeled trees]\label{def:rooted_unlabeled_trees}
	We recursively define the set $\cX_d$ of \emph{finite rooted unlabeled trees of depth at most $d \geq 0$}. For $d = 0$, $\cX_d$ contains the trivial tree reduced to its root node, denoted $\bigcdot$. For $d \geq 1$, having defined $\cX_0, \ldots, \cX_{d-1}$, we define $\cX_d$ as follows: a finite rooted unlabeled tree $t \in \cX_d$ consists in a sequence $\{N_{\tau} \}_{\tau \in \cX_{d-1}}$ of non-negative integers with finite support, that is such that 
	$$ \card{\left\{\tau \in \cX_{d-1}, N_{\tau} \neq 0 \right\}}  < \infty,$$
	where $N_{\tau}$ is the number of children of the root in $t$ which subtrees are equal to $\tau$.
	
	{Throughout the study, unless stated otherwise, we only consider finite rooted unlabeled trees and will sometimes omit the adjectives `finite', `rooted' or `unlabeled' in the sequel for brevity.}  
\end{definition} 

\begin{remark}
	With the previous definition, equality between two rooted unlabeled trees $t := \{N_{\tau} \}_{\tau \in \cX_{d-1}}$ and $t' := \{N'_{\tau} \}_{\tau \in \cX_{d-1}}$ is defined as $N_{\tau} = N'_{\tau} $ for all $\tau \in \cX_{d-1}$.
\end{remark}

\begin{remark}
	Denoting one-to-one correspondence by $\simeq$, we remark that $\cX_1 \simeq \dN$, and that more generally for each $d \geq 1$, 
	$$ \cX_d \; \simeq \; \bigcup_{\ell \geq 0} \; \bigcup_{\{\tau_1, \ldots, \tau_\ell \} } \dN_+^{\ell} \ ,$$ 
	where the second union is over subsets of $\ell$ distinct elements in $\cX_{d-1}$.
	Hence, $\cX_d$ is countably infinite for all $d \geq 1$.
\end{remark}

\begin{definition}[Size of a rooted unlabeled tree]
	The \emph{size}, or \emph{number of nodes}, of a tree $t \in \cX_{d}$ is denoted by $\card{t}$ and defined recursively as follows. First, if $d=0$, we set $\card{\bigcdot}=1$. Then, for $d \geq 1$, writing $t = \{N_{\tau} \}_{\tau \in \cX_{d-1}}$, one has
	\begin{equation*}
		\card{t} = 1 + \sum_{\tau \in \cX_{d-1}} N_{\tau} \cdot \card{\tau} \,.
	\end{equation*}
\end{definition}

\begin{definition}[Depth of a rooted unlabeled tree]
	The \emph{depth} of a tree $t \in \cX_{d}$ is denoted by $\depth{t}$ and defined recursively as follows. First, if $d=0$, we set $\depth{\bigcdot}=0$. Then, for $d \geq 1$, writing $t = \{N_{\tau} \}_{\tau \in \cX_{d-1}}$, one has
	\begin{equation*}
		\depth{t} = 1 + \max\left\{ \one_{N_\tau \geq 1} \cdot \depth{\tau}, \; \tau \in \cX_{d-1} \right\} \,.
	\end{equation*}
\end{definition}

\begin{remark}\label{remark:Xd_increasing}
	Note that for $t \in \cX_{d}$, $\depth{t} \leq d$. Note also that for all $d< d'$ there is a canonical injective mapping $\phi_{d\to d'}$ from $\cX_d$ to $\cX_{d'}$ defined in the following recursive manner. For $t=\{N_\tau\}_{\tau\in\cX_{d-1}}\in\cX_d$, let
	$\phi_{d\to d'}(t)=\{N'_{\tau'}\}_{\tau'\in\cX_{d'-1}}$
	where 
	$$
	N'_{\tau'}= \begin{cases}
		N_\tau& \hbox{if there is $\tau\in \cX_{d-1}$ such that } \tau'=\phi_{d-1\to d'-1}(\tau),\\
		0&\hbox{otherwise,}
	\end{cases} 
	$$
	the recursion being initialized by defining $\phi_{0\to d}(\bigcdot)=\{N_\tau\}_{\tau\in\cX_{d-1}}$ with $N_\tau\equiv 0$.
	
	It is readily seen that $t\in\cX_{d'}$ is in the image of $\phi_{d\to d'}$ if and only if $\depth{t}\leq d$. In turn, assuming that $d=\depth{t}\leq d'$, one can define the canonical representative $\rep{t}$ of $t$ by setting $\rep{t}:=\phi_{d\to d'}^{-1}(t)$. 
\end{remark}






\subsubsection{Formal power series}
If $f$ is a formal power series in the variable $x$, we denote $[x^n]f(x)$ the coefficient of the monomial $x^n$ in $f$, i.e. if $f(x)=\sum_{n \geq 0}a_n x^n$ then $[x^n]f(x) := a_n$.

If $f$ is a formal power series in $m$ variables $(x_1, \ldots, x_m)$, and $\ell := (\ell_1, \ldots, \ell_m)$ is a tuple of non negative integers, we use the shorthand notation $[x^\ell]f(x_1, \ldots, x_m)$ for $[x_1^{\ell_1} \cdots x_m^{\ell_m}]f(x_1, \ldots, x_m)$. 

Throughout the paper we will often consider families indexed by a countably infinite set $\cZ$, and in particular use the same shorthand $[x^\ell]$ for $[\prod_{z \in \cZ} x_z^{\ell_z}]$, where $x=\{x_z\}_{z \in \cZ}$ is a family of formal variables, and $\ell = \{\ell_z\}_{z \in \cZ}$ a family of non-negative integers; in such occurences only a finite number of $\ell_z$ will be non-zero, the definition thus reduces to the finite-dimensional one by taking $x_z=0$ whenever $\ell_z=0$. This \emph{finite support} property will also make summations over $z \in \cZ$ of functions of $\ell_z$ well-defined.

By convention $\left\{u_z \right\}_{z \in \cZ}$ will stand for non-negative integer sequences $\left\{u_z \right\}_{z \in \cZ}$ with finite support, hence from the definition of $\cX_d$ the sum $\sum_{t \in \cX_d}$ will be equivalently denoted $\sum_{\left\{N_\tau \right\}_{\tau \in \cX_{d-1}}}$. 

\subsubsection{Cardinality of unlabeled trees with given size and depth}

\begin{definition}[Number of trees]\label{def:An_Adn}
	For $n \geq 1$, let us define
	\begin{equation}\label{eq:def:An}
		A_n := 
		\card{ \left\{ t \in \cX_{n-1}, \; \card{t}=n \right\} } \, ,
	\end{equation} that is $A_n$ is the number of (distinct) unlabeled rooted trees of size $n$. For $d \geq 0$, we furthermore define
	\begin{equation}\label{eq:def:Adn}
		A_{d,n} := \card{ \left\{ t \in \cX_d, \; \card{t}=n \right\} } \, ,
	\end{equation} that is $A_{d,n}$ is the number of (distinct) unlabeled rooted trees of size $n$ and depth at most $d$. 
\end{definition}

\begin{remark}
	It follows from Remark \ref{remark:Xd_increasing} together with $\depth{t} \leq \card{t}-1$ 
	that for all $n$  the sequence $(A_{d,n})_{d \geq 0}$ is non-decreasing and converges to $A_n$.
\end{remark}

We now state a celebrated result by Otter \cite{Otter48}:

\begin{proposition}[Asymptotic number of unlabeled trees, \cite{Otter48}]\label{prop:otter}
	One has
	\begin{equation}\label{eq:theorem:otter}
		A_n \underset{n \to \infty}{\sim} \frac{D}{n^{3/2}} \left( \frac{1}{\alpha} \right)^n,
	\end{equation} for some $D>0$, where $\alpha \in (0,1)$ is the Otter constant, numerically $\alpha = 0.3383219...$ As a consequence the power series
	\begin{equation}\label{eq:prop:A_dn:phi}
		\Phi(x) := \sum_{n \geq 1} A_{n} x^{n-1} 
	\end{equation} 
	has a radius of convergence equal to $\alpha$, and $\Phi(\alpha)<\infty$.
\end{proposition}

An object of interest in the sequel is the generating function of the sequence $\{A_{d,n}\}_{n\geq 1}$. The following proposition defines this generating function $\Phi_d$ and states some elementary properties of $\Phi_d$.

\begin{proposition}[Control of the generating function of $\{A_{d,n}\}_{n\geq 1}$]\label{prop:A_dn}
	For all $d \geq 0$, define the power series
	\begin{equation}\label{eq:prop:A_dn:phid}
		\Phi_d(x) := \sum_{n \geq 1} A_{d,n} x^{n-1} \, .
	\end{equation} Then for all $x\geq 0$, 
	\begin{equation}\label{eq:prop:A_dn:cv_phi}
		\Phi_d(x) \underset{d \to \infty}{\longrightarrow} \Phi(x) \, .
	\end{equation}
	Moreover, for all $d \geq 0$ and $t \in [0,1)$, there exists $A=A(d,t)$ such that 
	\begin{equation}\label{eq:prop:A_dn}
		\forall x \in [0,t], \, \Phi_d(x) \leq A \, .
	\end{equation}
\end{proposition}

\begin{proof}[Proof of Proposition \ref{prop:A_dn}]
	For each $n \geq 1$  the sequence $A_{d,n}$ is non-decreasing in $d$, and such that $A_{d,n} \underset{d \to \infty}{\longrightarrow} A_n$ (the limit being actually reached for $d \geq n-1$). The limit \eqref{eq:prop:A_dn:cv_phi} thus follows from the monotone convergence theorem. 
	
	
	We now establish a recursion property on the $\Phi_d$. Write
	\begin{flalign*}
		\Phi_{d+1}(x) & = \sum_{n \geq 1} x^{n-1} A_{d+1,n} = \sum_{t \in \cX_{d+1}} x^{\card{t}-1}  = \sum_{\set{N_\tau}_{\tau \in \cX_{d}}} x^{\sum_{\tau \in \cX_{d}} N_\tau \card{\tau}} 
		\\ 
		& \overset{(a)}{=} \prod_{\tau \in \cX_{d}} \sum_{N_\tau \geq 0} (x^{\card{\tau}})^{N_\tau}
		= \prod_{\tau \in \cX_{d}} \frac{1}{1-x^{\card{\tau}}} = \exp \left( - \sum_{\tau \in \cX_{d}} \log(1-x^{\card{\tau}})  \right) \\
		& = \exp\left(\sum_{\tau \in \cX_{d}} \sum_{j \geq 1} \frac{(x^{j})^\card{\tau}}{j} \right) \overset{(b)}{=} \exp\left(\sum_{j \geq 1} \frac{x^j}{j} \Phi_d(x^{j}) \right) \, .
	\end{flalign*}
	All terms being positive, the interchanges in steps $(a)$ and $(b)$ are legitimate. Equation \eqref{eq:prop:A_dn} is then easy to establish by induction with this last formula. For $d=0$, $A_{d,n} = \one_{n=1}$, hence $\Phi_0(x)=1$ and \eqref{eq:prop:A_dn} holds with $A=1$. Assume that \eqref{eq:prop:A_dn} holds at depth $d$ for all $t \in [0,1)$ with constant $A(d,t)$. Then, by the previous computation, for all $t \in [0,1)$ and $x \in [0,t]$, since $x^j \in [0,t]$ for all $j \geq 1$ we have
	\begin{flalign*}
		\Phi_{d+1}(x) & = \exp\left(\sum_{j \geq 1} \frac{x^j}{j} \Phi_d(x^{j}) \right) \leq \exp\left(\sum_{j \geq 1} \frac{x^j}{j} A(d,t) \right) \\
		& = \exp\left(- A(d,t) \log(1-x)  \right) = \left(\frac{1}{1-x}\right)^{A(d,t)} \leq \left(\frac{1}{1-t}\right)^{A(d,t)} =: A(d+1,t)\, .
	\end{flalign*} Thus, \eqref{eq:prop:A_dn} holds at depth $d+1$. In particular, \eqref{eq:prop:A_dn} shows that $\Phi_d$ has a radius of convergence not smaller than $1$. As a matter of fact the radius of convergence of $\Phi_d$ is equal to 1 for all $d\geq 1$: indeed, $A_{d,n} \geq 1$ for all $n\geq 1$ (consider the tree made of the root and $n-1$ children), hence $\Phi_d(x)=\infty$ for all $x \geq 1$.
\end{proof}

\subsubsection{Models of random trees}
We now define the models of random trees considered in the study. 

{Let us start with the Galton-Watson distribution of offspring $\Poi(\mu)$ with $\mu>0$ at depth at most $d$. The most usual definition in terms of labeled trees correspond to construct recursively from the root a tree up to depth $d$ where each vertex has an independent number of offsprings drawn from $\Poi(\mu)$. A canonical labelling is obtained by assigning the label $\varnothing$ to the root, the $i$-th offspring of a vertex $v$ being labeled $vi$ (see Fig.~\ref{fig_def_trees} for an example). Note that conditional on the degree of the root, each of its offspring is the root of an independent labeled Galton-Watson tree of depth at most $d-1$. This random construction induces a probability distribution on unlabeled trees when identifying the trees that are related by parenthood-preserving isomorphisms, denoted $\GWmu_{d}$ and formally defined below. After unlabelling, each offspring of the root is the root of an unlabeled tree $\tau \in \cX_{d-1}$ with probability $\GWmu_{d-1}(\tau)$; conditionally on the degree $\ell$ of the root, the number of copies $\{N_\tau\}$ of the unlabeled trees of depth $d-1$ present under the root is thus drawn according to the multinomial distribution of parameters $(\ell;\{\GWmu_{d-1}(\tau)\})$. When averaging such a multinomial distribution with $\ell$ drawn from a Poisson distribution $\Poi(\mu)$ one obtains a family of independent Poisson random variables with parameters $\{\mu \GWmu_{d-1}(\tau)\}$, hence:} 
\begin{definition}[Galton-Watson trees with Poisson offspring]\label{def:GW_trees}
	Let $\mu>0$. For $d=0$, $\GWmu_d$ is the Dirac mass at the trivial tree $\bigcdot \in \cX_0$. For $d \geq 1$, a tree $t = \{N_{\tau} \}_{\tau \in \cX_{d-1}} \sim \GWmu_d$ is sampled as follows: for all $\tau \in \cX_{d-1}$, $N_{\tau} \sim \Poi(\mu \GWmu_{d-1}(\tau))$ independently from everything else. Note that since the Poisson variables are independent, we have
	\begin{equation*}
		\sum_{\tau \in \cX_{d-1}} N_{\tau} \sim \Poi\left( \mu \sum_{\tau \in \cX_{d-1}} \GWmu_{d-1}(\tau) \right) = \Poi(\mu)
	\end{equation*}
	which is a.s. finite. Hence, we have $t \in \cX_d$ a.s.
	
	Throughout the study, we only consider Galton-Watson trees with Poisson offspring, which we will refer to as Galton-Watson trees for brevity.
\end{definition}

\begin{definition}[Null model $\dPl_d$]\label{def:null_model}
	The null distribution $\dPl_d$ on $\cX_d \times \cX_d$ of parameter $\lambda>0$ is  defined as the product $\GWl_{d} \otimes \GWl_{d}$: under the null model, the two trees are independent Galton-Watson trees with offspring $\Poi(\lambda)$.
\end{definition}

{We proceed with the definition of the correlated model of Galton-Watson trees, with marginal offspring $\Poi(\lambda), \lambda >0$ and correlation $s \in [0,1]$, which describes the local exploration of an aligned pair of correlated \ER graphs with these parameters~\cite{piccioli2021aligning}. For labeled trees, we recursively construct a colored (or multi-type) tree up to depth $d$ where each vertex has an independent number of offsprings linked to it by a blue (resp. red, bicolored) edge  drawn as independent $\Poi(\lambda(1-s))$ (resp. $\Poi(\lambda(1-s))$, $\Poi(\lambda s)$) random variables. One then obtains a pair of trees by including in the first (resp. second) one the blue and bicolored (resp. the red and bicolored) edges, and by keeping only the connected component of the root. This pair of trees are thus two labeled correlated Galton-Watson trees. The induced distribution for unlabeled trees, denoted $\dPls_d$, is deduced with the same reasoning as the one that led to Definition~\ref{def:GW_trees}, and formally defined below.} 

\begin{definition}[Correlated model $\dPls_d$]\label{def:correlated_model}
	The correlated model $\dPls_d$ on $\cX_d \times \cX_d$ with parameters $\lambda>0$ and $s \in [0,1]$ is defined as follows. For $d=0$, $\dPls_{0}= \dPl_0$. 
	For $d \geq 1$, a pair of trees $(t,t') = (\{N_{\tau} \}_{\tau \in \cX_{d-1}}, \{N'_{\tau} \}_{\tau \in \cX_{d-1}}) \sim \dPls_d$ is sampled as:
	\begin{equation}\label{eq:def_model_1}
		N_{\tau} := \Delta_\tau + \sum_{\tau' \in \cX_{d-1}} M_{\tau, \tau'} \quad \mbox{and} \quad
		N'_{\tau'} := \Delta'_{\tau'} + \sum_{\tau \in \cX_{d-1}} M_{\tau, \tau'} \, ,
	\end{equation} where $\{\Delta_{\tau}\}$, $\{\Delta'_{\tau}\}$ and $\{M_{\tau, \tau'}\}$ are families of independent random variables of Poisson laws with parameters $\lambda (1-s) \GWl_{d-1}(\tau)$ for the first two, and $\lambda s \dPls_{d-1}(\tau,\tau')$ for the last one. Note that for all $d$, $\dPls_d = \dPl_d$ if $s=0$.
\end{definition}

We will denote by $\dEl_d$ (respectively $\dEls_d$) the expectation {under} $\dPl_d$ (respectively $\dPls_d$).

\begin{remark}
 {The correlated model $\dPls_d$ defined in Definition \ref{def:correlated_model} is equivalent to the one of \cite{GMLTrees2021journal} formulated in terms of random augmentations of Galton-Watson trees.}
\end{remark}

\subsection{Main results}
The hypothesis test of interest can be formalized as follows: given the observation of a pair of (rooted, unlabeled, finite) trees $(t,t')$ in $\cX_d \times \cX_d$, we want to test
\begin{equation}
	\label{eq:test_hypotheses_p}
	\cH_0 = \mbox{"$t,t'$ are drawn under $\dPl_d$"} \quad \mbox{versus} \quad \cH_1 = \mbox{"$t,t'$ are drawn under $\dPls_d$"}.
\end{equation} 

\begin{remark}\label{remark:types_de_tests}
	In statistical detection theory, commonly considered asymptotic properties of tests are
	\begin{itemize}
		\item \emph{strong detection}, i.e. tests $\cT_d: \cX_d \times \cX_d \to \left\lbrace 0,1 \right\rbrace$  that verify
		\begin{equation*}
			\underset{d \to \infty}{\lim} \left[\dPl_d\left( \cT_d(t,t') = 1 \right) + \dPls_d\left( \cT_d(t,t') = 0 \right)\right] = 0,
		\end{equation*}
		\item \emph{weak detection}, i.e. tests $\cT_d: \cX_d \times \cX_d \to \left\lbrace 0,1 \right\rbrace$ that verify
		\begin{equation*}
			\underset{d \to \infty}{\limsup} \left[\dPl_d\left( \cT_d(t,t') = 1 \right) + \dPls_d\left( \cT_d(t,t') = 0 \right)\right] <1 \, .
		\end{equation*}
	\end{itemize} In words, strong detection corresponds to correctly discriminating with high probability between $\dPl_d$ and $\dPls_d$, whereas weak detection corresponds to strictly outperforming random guessing. It is well-known that the likelihood ratio test achieves the minimal value of $\dPl_d\left( \cT_d(t,t') = 1 \right) + \dPls_d\left( \cT_d(t,t') = 0 \right)$, which is the sum of type I and type II error, and that this minimal value is given by $1 - \DTV(\dPl_d,\dPls_d)$, where $$\DTV(\dPl_d,\dPls_d) := \frac{1}{2} \sum_{(t,t') \in \cX_d^2} |\dPl_d(t,t') - \dPls_d(t,t')| $$ denotes the total variation distance
	between $\dPl_d $ and $\dPls_d$. Hence, strong detection (resp. weak detection) holds if and only if $\, \DTV(\dPl_d,\dPls_d) \underset{d \to \infty}{\longrightarrow} 1$ (resp. $\underset{d \to \infty}{\liminf} \, \DTV(\dPl_d,\dPls_d) > 0$).
	
	We now argue that these are not the right notions for our problem. For this, consider the event
	$$ B_d := \set{(t,t')=(\bigcdot,\bigcdot)} \, .$$
	Note that for all $d>0$, $\dPl_d(B_d)=e^{-2\lambda}$ and $\dPls_d(B_d)=e^{-2\lambda+\lambda s}$.
	\begin{itemize}
		\item First, let us prove that strong detection never holds. For all $d>0$,
		\begin{flalign*}
			\DTV(\dPl_d,\dPls_d) & = \frac{1}{2}|\dPl_d(B_d) - \dPls_d(B_d)| + \frac{1}{2} \sum_{(t,t') \in \cX_d^2 \setminus \set{(\bigcdot,\bigcdot)}} |\dPl_d(t,t') - \dPls_d(t,t')| \\
			& = \frac{1}{2}(e^{-2\lambda+\lambda s}-e^{-2\lambda}) + \frac{1}{2} \sum_{(t,t') \in \cX_d^2 \setminus \set{(\bigcdot,\bigcdot)}} |\dPl_d(t,t') - \dPls_d(t,t')| \\
			& \leq \frac{1}{2}(e^{-2\lambda+\lambda s}-e^{-2\lambda}) + \frac{1}{2}(1-\dPl_d(B_d)) +  \frac{1}{2}(1-\dPls_d(B_d)) \\
			& \leq 1 - e^{-2\lambda} \, .
		\end{flalign*} This bound is uniform in $d$ and shows that strong detection never holds.

		\item Second, weak detection is always achievable as soon as $s>0$: indeed, one has 
		\begin{flalign*}
			\DTV(\dPl_d,\dPls_d) & \geq \frac{1}{2}|\dPl_d(B_d) - \dPls_d(B_d)| \\
			& \geq \frac{1}{2} e^{-2\lambda}(e^{\lambda s}-1) \, ,
		\end{flalign*} this uniform bound being positive as soon as $s>0$, so that weak detection holds.
		
	\end{itemize}
	Finally, a test of tree correlation yields efficient algorithms for graph alignment in the associated sparse correlated  \ER model if it achieves a positive power (non-vanishing alarm detection) and a vanishing type I (false alarm) error. Indeed the candidate vertex pairs returned by the algorithm will then contain i) a non-negligible fraction of correctly matched pairs by the first property, and ii) a negligible fraction of incorrectly matched pairs by the second property.

\end{remark}

We thus focus on the existence of a \textit{one-sided test}, that is a test $\cT_d: \cX_d \times \cX_d \to \left\lbrace 0,1 \right\rbrace$ such that $\cT_d$ chooses hypothesis $\cH_0$ under $\dPl_d$ with probability $1-o(1)$, and chooses $\cH_1$ under $\dPls_d$ with non-vanishing probability. According to the Neyman-Pearson Lemma, optimal tests are based on the \emph{likelihood ratio} $L_d$ of the distributions under the distinct hypotheses $\dPls_{d}$ and $\dPl_{d}$, given for a pair of trees $(t,t')$ by
\begin{equation*}
	L_d(t,t') := \frac{\dPls_d(t,t')}{\dPl_d(t,t')} \, .
\end{equation*} 

Note that provided $\lambda >0$, $\dPl_d(t,t')>0$ for all $t,t' \in \cX_d$, so that there is no division by zero in the above definition.

We define 
\begin{equation}\label{eq:def_KLd}
	\KLls_d := \KL (\dPls_d\Vert\dPl_d)= \dEls_d \left[ \log(L_d) \right] \, ,
\end{equation} with $\KL$ denoting the Kullback-Leibler divergence.

The following Lemma gives an easy necessary condition for feasibility of one-sided detection. It is substantially the same as Proposition 3.2 in \cite{GMLTrees2021journal}. We however give a proof hereafter for the sake of {self-consistency}.
\begin{lemma}\label{lemma:one_sided_implies_divergence}
	If $\lambda, s$ are such that there exists a one-sided test, then $ \KLls_d \underset{d \to \infty}{\longrightarrow} + \infty$. 
\end{lemma}

\begin{proof}[Proof of Lemma \ref{lemma:one_sided_implies_divergence}]
	Assume that there exists a one-sided test. Then for every $d \geq 0$ there is an event $A_d \subset \cX_d^2$ such that $\dPl_d(A_d) \underset{d \to\infty}{\longrightarrow} 0$ and $\eps := \liminf_{d \to\infty} \dPls_d(A_d) >0$. Elementary properties of the Kullback-Leibler divergence entail
	\begin{align*}
		\KLls_d & \geq \dPls_d(A_d) \log\frac{\dPls_d(A_d)}{\dPl_d(A_d)}  + (1-\dPls_d(A_d)) \log\frac{1-\dPls_d(A_d)}{1-\dPl_d(A_d)} \\
		& = - \dPls_d(A_d)\log \dPl_d(A_d)   + \dPls_d(A_d) \log \dPls_d(A_d) \\ & \quad \quad + (1-\dPls_d(A_d))\log(1-\dPls_d(A_d)) - \underbrace{(1-\dPls_d(A_d))\log(1-\dPl_d(A_d))}_{\leq 0} \\
		& \geq - \dPls_d(A_d)\log \dPl_d(A_d) + g(\dPls_d(A_d)) \, ,
	\end{align*} 
	where for $x \in [0,1]$, $g$ is defined by $g(x) :=  x \log(x) + (1-x) \log(1-x)$. Function $g$ is minimal at $x=1/2$ and $g(1/2)=-\log(2)$, which gives the final bound $$\liminf_{d \to\infty}  \KLls_d \geq \eps \lim_{d \to\infty}(  - \log \dPl_d(A_d) )- \log 2 = + \infty \, .$$
\end{proof}
We now state the main results of this paper. Let $\alpha$ be the Otter constant introduced in Proposition \ref{prop:otter}.

\begin{theorem}[Negative result]\label{thm:negative_result}
	If $s \leq \sqrt{\alpha}$, then ${\limsup_{d\to\infty} \KL(\dPls_d \, || \, \dPl_d) < \infty}$ for all $\lambda>0$. Hence, in view of Lemma \ref{lemma:one_sided_implies_divergence}, one-sided detection is impossible.
\end{theorem}

{For this negative result, we will actually prove a slightly stronger result, see Section \ref{section:proof_negative_result}, namely that the $\chi^2$-distance between the two distributions, $\chi^2(\dPls_d \, || \, \dPl_d) := \dEls_d \left[ L_d \right]-1$, is bounded uniformly in $d$.}

\begin{theorem}[Positive result]\label{thm:positive_result}
	If $s > \sqrt{\alpha}$, then there exists $\lambda(s)>0$ such that for all $\lambda \geq \lambda(s)$, one-sided detection is feasible.
\end{theorem}

For the positive result, we are going to look at the high degree limit $\lambda \to \infty$: this intuitively simplifies the problem, since the Poisson degree distribution becomes Gaussian. In the case where $d>1$, if $(t,t')$ are viewed individually as trees of depth $d$, the notion of 'gaussianity' is however less clear. The strategy is hence to operate a 'change of basis' in which the limiting object will be easier to define and more appealing, see Section \ref{subsection:gaussian_approx}.

\begin{remark}
	We present on the left panel of Figure \ref{fig:phase_diagram_new} a sketch of the conjectured phase diagram for the problem of one-sided detection in random trees, dividing the plane of parameters $(\lambda,s)$ in regions where this task is possible or not. The results of the present paper show that $s \leq \sqrt{\alpha}$ lies in the impossible phase, while for any $s > \sqrt{\alpha}$ the possible phase appears for sufficiently large $\lambda$. We expect more precisely that increasing the correlation parameter $s$ can only bring from the impossible to the possible phase, in other words that there exists a phase transition line $s_{\rm c}(\lambda)$, drawn on the left panel of Fig.~\ref{fig:phase_diagram_new}, such that the inference task is feasible (resp. unfeasible) if $s<s_{\rm c}(\lambda)$ (resp. $s>s_{\rm c}(\lambda)$). It was shown in \cite{GMLTrees2021journal} that the regime $\lambda s <1$ is contained in the impossible phase, and that one-sided tests exist in some parts of the regime $\lambda s >1$; for $\lambda \in [1,1.178]$ these upper and lower bounds coincide, implying that $s_{\rm c}(\lambda) = 1/\lambda$ in some interval of $ \lambda$ slightly above 1. Since the transition line cannot cross $\sqrt{\alpha}$, there exists a value $\lambda_{\rm t} \geq 1.178$ (marked with a black dot on the figure) such that $s_{\rm c}(\lambda) = 1/\lambda$ when $\lambda \in [1,\lambda_{\rm t}]$, and 
	$s_{\rm c}(\lambda) > 1/\lambda$ for $\lambda > \lambda_{\rm t}$.
	
	In the right panel of the figure we consider instead the problem of partial recovery of the hidden permutation in random graph alignment. We divide the plane of parameters in three regions, the impossible phase where information-theoretic bounds forbid the recovery of a non-vanishing fraction of aligned vertices, the easy phase in which an algorithm achieves this goal in a time growing polynomially with the number of vertices, and the hard phase in which this task is information-theoretically possible but no efficient algorithm is known for it. The negative and positive results of \cite{ganassali2021impossibility} and \cite{Ding22} respectively imply that the boundary of the impossible phase is the curve $\lambda s=1$. According to Theorem~\ref{theorem:MPalign} the regime $s>s_{\rm c}(\lambda)$ where one-sided detection in the tree problem is feasible is included in the easy phase of the graph recovery problem. Note that \cite{GMLTrees2021journal} conjectures that this inclusion is actually an equality, i.e. that if one-sided correlation detection in trees fails, then no polynomial-time algorithm achieves partial graph alignment. Under this conjecture the hard phase would fill the remaining of the phase diagram, i.e. the region $\left\lbrace (\lambda, s) : \lambda s>1, s < s_{\rm c}(\lambda) \right\rbrace$, which from Theorem \ref{thm:negative_result} contains the domain $\left\lbrace (\lambda, s) : \lambda s>1, s \leq \sqrt{\alpha} \right\rbrace$.
	
\end{remark}



\begin{figure}
	\begin{center}
		\begin{tikzpicture}[scale=.8]
			\draw[-latex] (0,0) -- (0,3.5) node [left] {$s$};
			\draw[-latex] (0,0) -- (6.5,0) node [below] {$\lambda$};
			\draw (6.3,3) -- (0,3) node [left] {$1$};
			\draw (1.5,3) to[out=-65,in=120] (1.8,2.45);
			\draw (1.8,2.45) to[out=-30,in=180] (6.,1.6);
			\draw (1.8,1) node {impossible};
			\draw (4.2,2.4) node {possible};
			\draw[dashed] (5,1.55)--(6.7,1.55) node [right] {$\sqrt{\alpha}$};
			\draw (1.5,3) node [above] {$1$};
			\fill[black] (1.8,2.45) circle (1.2pt);
			\begin{scope}[xshift=10cm]
				\draw[-latex] (0,0) -- (0,3.5) node [left] {$s$};
				\draw[-latex] (0,0) -- (6.5,0) node [below] {$\lambda$};
				\draw (6.3,3) -- (0,3) node [left] {$1$};
				\draw (1.5,3) to[out=-65,in=175] (6.3,.2);
				\draw (1.8,2.45) to[out=-30,in=180] (6.,1.6);
				\draw (1.8,1) node {impossible};
				\draw (4.2,2.4) node {easy};
				\draw (5.5,1) node {hard};
				\draw[dashed] (5,1.55)--(6.7,1.55) node [right] {$\sqrt{\alpha}$};
				\draw (1.5,3) node [above] {$1$};
				\fill[black] (1.8,2.45) circle (1.2pt);
			\end{scope}
		\end{tikzpicture}
	\end{center}
	\caption{A sketch of the conjectured phase diagram for (left) the one-sided hypothesis testing problem on trees, (right) partial recovery for random graph alignment.}
	\label{fig:phase_diagram_new}
\end{figure}



\section{The impossible phase for $s \leq \sqrt{\alpha}$}








\subsection{Decomposition of the likelihood ratio}

The key result of this section is the following

\begin{theorem}[Decomposition of $L_d$]\label{theorem:eigendecomposition}
	For all $\lambda>0, d \geq 0$, there exists a collection $\{\fl_{d,\beta} \}_{\beta \in \cX_d}$ with $\fl_{d,\beta} : \cX_d \to \dR$, such that for all $s \in [0,1]$, 
	
	\begin{equation}\label{eq:theorem:eigendecomposition}
		\forall t,t' \in \cX_d, \; L_d(t,t') = \sum_{\beta \in \cX_d} s^{|\beta| -1} \fl_{d,\beta}(t) \fl_{d,\beta}(t') \, . 
	\end{equation}
	Moreover, the $\fl_{d,\beta}$ are independent of $s$ and verify the following properties:
	\begin{itemize}
		\item Value at the trivial tree: 
		\begin{equation}\label{eq:theorem:eigendecomposition_trivial_tree}
			\forall t \in \cX_d, \; \fl_{d, \bigcdot}(t) = 1 \, ,
		\end{equation}
		\item Orthogonality: 
		\begin{equation}\label{eq:theorem:eigendecomposition_orthogonality}
			\forall \beta, \beta' \in \cX_d, \; \sum_{t \in \cX_d} \GWl_d(t) \fl_{d, \beta}(t) \fl_{d, \beta'}(t) = \one_{\beta = \beta'} \, ,
		\end{equation}
		
		\begin{equation}\label{eq:theorem:eigendecomposition_orthogonality_2}
			\forall t, t' \in \cX_d, \; \sum_{\beta \in \cX_d}  \fl_{d, \beta}(t) \fl_{d, \beta}(t') = \frac{\one_{t = t'}}{\GWl_d(t)} \, .
		\end{equation}

		\item Limit of higher-order mixed moments: for $n \geq 2$, $d \geq 1$ and $\beta^{(1)} = \{\beta^{(1)}_\gamma\}_{\gamma \in \cX_{d-1}}$, $\ldots, \beta^{(n)} = \{\beta^{(n)}_\gamma\}_{\gamma \in \cX_{d-1}} \in \cX_{d}$, one has
		\begin{multline}\label{eq:theorem:eigendecomposition_mixed_products}
			\sum_{t \in \cX_{d}} \GWl_{d}(t) \fl_{d, \beta^{(1)} }(t) \cdots \fl_{d, \beta^{(n)} }(t) \underset{\lambda \to \infty}{\longrightarrow} \prod_{\gamma \in \cX_{d-1}}\sqrt{\prod_{i=1}^n \beta^{(i)}_\gamma!} \left[ x^{\beta^{(1)}_\gamma}_1 \cdots \, x^{\beta^{(n)}_\gamma}_n \right] e^{\sum_{1 \leq i <j \leq n}x_i x_j} \, .
		\end{multline}
	\end{itemize}
\end{theorem}

\begin{remark}
	Note that the right hand side of \eqref{eq:theorem:eigendecomposition_mixed_products} simplifies for $n=2$ to $\one_{\beta^{(1)}=\beta^{(2)}}$; \eqref{eq:theorem:eigendecomposition_orthogonality} shows that for $n=2$ the equality in \eqref{eq:theorem:eigendecomposition_mixed_products} holds for all $\lambda$, not only in the limit $\lambda \to \infty$, which is nevertheless required for $n>2$. Note also that \eqref{eq:theorem:eigendecomposition_trivial_tree} combined with \eqref{eq:theorem:eigendecomposition_orthogonality} implies the following first moment condition: 
	\begin{equation*}
		\forall \beta \in \cX_d, \; \sum_{t \in \cX_d} \GWl_d(t) \fl_{d, \beta}(t) = \one_{\beta = \bigcdot}.
	\end{equation*}
\end{remark}

\begin{remark}
	Introducing the (infinite-dimensional) matrices $F$ and $D$ with indices in $\cX_d$ as $F(\beta,t) := \fl_{d,\beta}(t)$ and $D(\beta,\beta') := s^{|\beta|-1} \one_{\beta=\beta'} $ one can write \eqref{eq:theorem:eigendecomposition} as $L_d=F^T D F$, with $F$ depending on $\lambda$ but not on $s$, while the diagonal matrix $D$ depends on $s$ but not on $\lambda$. Furthermore, with $G(t,t'):=\GWl_d(t) \one_{t=t'}$ the orthogonality conditions \eqref{eq:theorem:eigendecomposition_orthogonality} and \eqref{eq:theorem:eigendecomposition_orthogonality_2} read $F G F^T = \one$ and $F^T F = G^{-1}$ respectively; with $\widehat{F}:=F G^{1/2}$ they become $\widehat{F} \widehat{F}^T = \widehat{F}^T \widehat{F} = \one$, these two relations would thus be equivalent for finite dimensional matrices. Rewriting \eqref{eq:theorem:eigendecomposition} as $G^{1/2} L_d G^{1/2} = \widehat{F}^{-1} D \widehat{F}$ reveals that we have achieved a diagonalization not exactly of $L_d$ but rather of $G^{1/2} L_d G^{1/2}$; with a slight abuse of vocabulary we shall call the $\fl_{d,\beta}$ "eigenvectors of $L_d$", keeping implicit this slight difference.
\end{remark}

\begin{remark}
	In~\cite{GMLTrees2021journal} a Markov chain on the space of trees was introduced, with a transition kernel that reads in the above notations $M(s):=L_d G$, emphasizing its dependency on the correlation parameter $s$; its matrix elements can be easily seen to be $M(s)(t,t')=\dPls_d(t'|t)$. The above results show that $M(s)=F^T D(s) F G = (FG)^{-1} D(s) FG$, i.e. that $M(s)$ can be diagonalized with eigenvectors depending only on $\lambda$, and eigenvalues depending only on $s$. The semi-group property $M(s) M(s')=M(s s')$ derived as Proposition 2.1 in~\cite{GMLTrees2021journal} follows directly from the properties of the diagonalization stated here, in particular the obvious multiplication rule $D(s) D(s')=D(s s')$.
\end{remark}

\begin{remark}\label{remark:charlier}
	One can view the $\fl_{d,\beta}$ as orthogonal polynomials on the space of unlabeled trees $\cX_d$, where the notion of orthogonality is with respect to the Galton-Watson measure $\GWl_d$ -- see \eqref{eq:theorem:eigendecomposition_orthogonality}. The proof below will indeed show that the dependency on $t=\{N_{\tau} \}_{\tau \in \cX_{d-1}}$ of $\fl_{d,\beta}(t)$ is polynomial in the entries $N_{\tau}$, with coefficients depending on the indexing tree $\beta$. 
	
	In particular the $\fl_{d, m}(\ell)$ for $d=1$, indexed by $m, \ell \in  \dN \simeq \cX_1 $, are given by
	\begin{equation*}
		\fl_{1, m}(\ell) := \sqrt{m!} [x^m] e^{-x \sqrt{\lambda}} \left( 1+\frac{x}{\sqrt{\lambda}}\right)^\ell \, ,
	\end{equation*} see equation \eqref{eq:fl1} in the proof. These functions are known as Charlier polynomials of degree $m$, and are orthogonal for the Poisson distribution. Theorem \ref{theorem:eigendecomposition} provides an infinite-dimensional extension of these polynomials, on trees of depth $d \geq 2$, that are orthogonal for the $\GWl_d$ distribution, consistent with $\GWl_1 \eqd \Poi(\lambda)$.
	Since under $\GWl_d$ the Poisson random variables $N_\tau$ are independent, a possible basis of orthogonal polynomials on $\cX_d$ could be built with products over $\tau \in \cX_{d-1}$ of Charlier polynomials of arguments $N_\tau$; it will be seen in the proof that the $\fl_{d,\beta}$ are not of this simple factorized form, because of the additional requirement that they are eigenvectors of $L_d$.
	
	Also, note that equation \eqref{eq:theorem:eigendecomposition} exhibits a duality between trees $t$ in $\cX_d$ and the trees $\beta \in \cX_d$. This duality turns out to be very helpful for analysis, as shown below, e.g. giving a nice space in which one can prove weak convergence results -- see Section \ref{subsection:gaussian_approx}.
\end{remark}

\begin{remark}
	The decomposition \eqref{eq:theorem:eigendecomposition} may look reminiscent of the low-degree polynomials approach to hypothesis testing problems (see e.g. \cite{Kunisky19,Mao21_counting}), in which the likelihood ratio is projected onto the subspace of restricted degree polynomials in the observations. This similarity hides some differences, in particular \eqref{eq:theorem:eigendecomposition} provides an exact expression of the likelihood ratio and not an approximation, and also the observations are here infinite-dimensional for $d\ge 2$, which makes the expansion basis less immediately apparent than for, say, the entries of the adjacency matrix of a graph. Nevertheless one can follow this idea to build approximate versions of the likelihood ratio by truncating the sum in \eqref{eq:theorem:eigendecomposition} to a well-chosen sub-family of trees. Preliminary numerical results suggest that it is possible in this way to construct tests that achieve one-sided detection in a non-empty part of the phase diagram, while being computable recursively in a more efficient fashion than the exact likelihood ratio.
\end{remark}

\begin{remark}
    {One can wonder if the results of Theorem \ref{theorem:eigendecomposition} can be extended to more general random correlated tree models, in particular to the local limit of the correlated random graph models described in~\cite{piccioli2021aligning}, and if the Otter's threshold $\sqrt{\alpha}$ on the correlation will play an universal role besides the Poisson Galton-Watson setting. These questions deserve further investigations, yet one should point out that Poisson distributions do enjoy some very specific properties that are crucially exploited in our computations and that do not generalize to arbitrary offspring distributions, for instance the fact that a Poisson random variable subsampled in a multinomial way gives rise to independent Poisson random variables. A related specificity of the \ER (for graphs) / Poisson (for their tree limit) setting is the possibility to clearly distinguish the parameter $s$, which controls the correlation between the two objects, from $\lambda$, which controls the marginal law of one of the two objects, in the sense that modifying $s$ does not affect the marginals. One can indeed obtain a pair of correlated \ER random graphs of parameters $(\lambda,s)$ by two independent subsamplings of the edges of a common parent random graph (keeping each edge independently with probability $s$), where the parent graph is a usual \ER random graph of average degree $\lambda/s$. A similar subsampling from a parent graph with an arbitrary (non-Poissonian) degree distribution (e.g. from the configuration model), dependent on $s$ and on other parameters, would yield a graph with a degree distribution that generically depends on $s$. The "correlation parameter" will thus also affect the marginal laws of the individual graphs, hence Otter's threshold $\sqrt{\alpha}$ may not have a particular meaning in this generic case.
    }
\end{remark}



\begin{proof}[Proof of Theorem \ref{theorem:eigendecomposition}]
	
	We will prove the decomposition \eqref{eq:theorem:eigendecomposition} as well as the properties \eqref{eq:theorem:eigendecomposition_trivial_tree}, \eqref{eq:theorem:eigendecomposition_orthogonality}, \eqref{eq:theorem:eigendecomposition_orthogonality_2} and \eqref{eq:theorem:eigendecomposition_mixed_products} by induction on $d$.\\
	
	\underline{\emph{Step 1: initialization at $d=1$}} The set $\cX_1$ of trees with depth at most 1 is in bijection with $\dN$, since such a tree $t$ is encoded by the number of children of its root, $\ell \in \dN$, in such a way that $|t|=\ell+1$. We will similarly represent the tree $\beta$ by the integer $m$, with $|\beta|=m+1$, and write $\fl_{1, m}(\ell)$ for the eigenvector parametrized in this way. Denote by $\hatdPls_1$ the characteristic function defined on $\dR^2$ by 
	$\hatdPls_1(k,k') := \dE\left[ e^{ik\ell + ik'\ell'}\right]$ where $(t,t') \sim \dPls_1$. Under this distribution $\ell = \Delta +M$ and $\ell'=\Delta' +M$ with $\Delta$, $\Delta'$ and $M$ three independent random variables of Poisson law with parameters $\lambda(1-s)$, $\lambda(1-s)$ and $\lambda s$ respectively, hence
	\begin{flalign*}
		\hatdPls_1(k,k') & = \exp\left[ \lambda(1-s) (e^{ik} + e^{ik'} - 2) + \lambda s (e^{i(k+k')}-1) \right]\\
		& = e^{\lambda(e^{ik}-1)} e^{\lambda(e^{ik'}-1)} \exp\left[ \lambda s (e^{ik}-1) (e^{ik'}-1) \right]\\
		& = \sum_{m \geq 0} s^m \frac{\lambda^m}{m!} (e^{ik}-1)^m e^{\lambda(e^{ik}-1)} (e^{ik'}-1)^m e^{\lambda(e^{ik'}-1)} = \sum_{m \geq 0} s^m \hatgl_{1,m}(k) \hatgl_{1,m}(k') \, , 
	\end{flalign*} with 
	\begin{flalign*}
		\hatgl_{1,m}(k) & := e^{- \lambda} \sqrt{\frac{\lambda^m}{m!}} e^{\lambda e^{ik}} (e^{ik}-1)^m
		= e^{- \lambda} \sqrt{m!} e^{\lambda e^{ik}} [x^m] e^{x\sqrt{\lambda}(e^{ik}-1)}
		\\ & = e^{- \lambda} \sqrt{m!}  [x^m] e^{-x\sqrt{\lambda}} e^{(\lambda + x\sqrt{\lambda})e^{ik}} \, .
	\end{flalign*} 
	
	We have an easy upper bound of the form $\left| \hatgl_{1,m}(k) \right| \leq \frac{C^m}{\sqrt{m!}}$, uniformly in $k$, which establishes the normal convergence of the series $\hatdPls_1(k,k')$ in the above. Hence, inverting the Fourier transform, we get
	\begin{flalign*}
		\dPls_1(t,t') & = \int_{[0,2 \pi]^2} \frac{dk dk'}{(2 \pi)^2} e^{-ik\ell - i k' \ell'} \hatdPls_1(k,k') = \sum_{m \geq 0} s^m \gl_{1,m}(\ell) \gl_{1,m}(\ell') \, ,
	\end{flalign*} with 
	\begin{flalign*}
		\gl_{1,m}(\ell) & := \int_{[0,2 \pi]} \frac{dk}{2 \pi} e^{-ik\ell} \hatgl_{1,m}(k) \\
		& = e^{- \lambda} \sqrt{m!}  [x^m] e^{-x\sqrt{\lambda}} \int_{[0,2 \pi]} \frac{dk}{2 \pi} e^{-ik\ell}   e^{(\lambda + x\sqrt{\lambda})e^{ik}} \\
		& = e^{- \lambda} \sqrt{m!}  [x^m] e^{-x\sqrt{\lambda}} \frac{(\lambda + x\sqrt{\lambda})^\ell}{\ell!}
		\, .
	\end{flalign*}
	We hence have that $L_1$ satisfies \eqref{eq:theorem:eigendecomposition} with
	\begin{flalign}\label{eq:fl1}
		\fl_{1, m}(\ell) := \sqrt{m!} [x^m] e^{-x \sqrt{\lambda}} \left( 1+\frac{x}{\sqrt{\lambda}}\right)^\ell \, .
	\end{flalign}
	
	Taking $m = 0$ in \eqref{eq:fl1} gives $\fl_{1,\bigcdot}(t) = 1$ and proves condition \eqref{eq:theorem:eigendecomposition_trivial_tree} at $d=1$. Let us now prove the orthogonality relations; note first that for all $m, m' \in \dN$,
	\begin{flalign}
		\sum_{\ell \geq 0} e^{- \lambda} \frac{\lambda^\ell}{\ell!}  \fl_{1, m}(\ell) \fl_{1, m'}(\ell) & = \sqrt{m! m'!} \sum_{\ell \geq 0} e^{- \lambda} \frac{\lambda^\ell}{\ell!}  [x^m y^{m'}] e^{-x \sqrt{\lambda} -y \sqrt{\lambda}} \left[\left( 1+\frac{x}{\sqrt{\lambda}}\right)\left( 1+\frac{y}{\sqrt{\lambda}}\right) \right]^\ell \nonumber \\
		& = \sqrt{m! m'!} [x^m y^{m'}] e^{-x \sqrt{\lambda} -y \sqrt{\lambda}}  \exp \left[ - \lambda + \lambda\left( 1+\frac{x}{\sqrt{\lambda}}\right)\left( 1+\frac{y}{\sqrt{\lambda}}\right) \right] \nonumber \\
		& = \sqrt{m! m'!} [x^m y^{m'}] e^{xy} = \one_{m=m'}\, , \label{eq:proof:orthogonality_d1}
	\end{flalign} which establishes \eqref{eq:theorem:eigendecomposition_orthogonality} for $d=1$. Previous computations are made rigorous by noticing that the formal series in \eqref{eq:fl1} has infinite radius of convergence, and appealing to Fubini's theorem. With the same arguments, a proof of  \eqref{eq:theorem:eigendecomposition_orthogonality_2} is obtained by writing
	\begin{flalign*}
		\fl_{1, m}(\ell) & = \sqrt{m!} [x^m] e^{-x \sqrt{\lambda}} \left( 1+\frac{x}{\sqrt{\lambda}}\right)^\ell \\
		& = \sqrt{m!} \ell! [x^m y^\ell] e^{-x \sqrt{\lambda} + y + \frac{x y}{\sqrt{\lambda}}} \\
		& = \frac{\ell!}{\sqrt{m!}} [y^\ell] e^y \left(\frac{y}{\sqrt{\lambda}} - \sqrt{\lambda} \right)^m \, ,
	\end{flalign*}
	hence
	\begin{flalign*}
		\sum_{m \geq 0} \fl_{1,m}(\ell) \fl_{1,m}(\ell') & = \ell!(\ell')! [x^\ell y^{\ell'}] e^{x+y} \sum_{m \geq 0}   \frac{1}{m!} \left(\frac{x}{\sqrt{\lambda}} - \sqrt{\lambda} \right)^m \left(\frac{y}{\sqrt{\lambda}} - \sqrt{\lambda} \right) ^m \\
		& = \ell!(\ell')! [x^\ell y^{\ell'}] e^{xy/\lambda +\lambda} = \frac{\one_{\ell = \ell'}}{e^{-\lambda} \lambda^\ell / \ell!} \, ,
	\end{flalign*}
	which proves \eqref{eq:theorem:eigendecomposition_orthogonality_2} for $d=1$.
	
	Generalizing the computation of \eqref{eq:proof:orthogonality_d1} to a product of $n \geq 2$ eigenvectors yields
	\begin{align}\label{eq:fl2}
		\sum_{\ell \geq 0} e^{- \lambda} \frac{\lambda^\ell}{\ell!} \fl_{1, m_1 }(\ell) \cdots \fl_{1, m_n }(\ell)  = \sqrt{\prod_{i=1}^{n} m_i!} \, \sum_{\ell \geq 0} e^{- \lambda} \frac{\lambda^\ell}{\ell!}  [x_1^{m_1} \, \cdots \, x_n^{m_n}] e^{-\sqrt{\lambda} \sum_{i=1}^{n}x_i}  \prod_{i=1}^{n} \left(1+\frac{x_i}{\sqrt{\lambda}}\right)^\ell \nonumber \\
		= \sqrt{\prod_{i=1}^{n} m_i!} \, [x_1^{m_1} \, \cdots \, x_n^{m_n}] \exp \left[- \lambda -\sqrt{\lambda} \sum_{i=1}^{n}x_i  + \lambda \prod_{i=1}^{n} \left(1+\frac{x_i}{\sqrt{\lambda}}\right) \right] \nonumber \\
		= \sqrt{\prod_{i=1}^{n} m_i!} \, [x_1^{m_1} \, \cdots \, x_n^{m_n}] \exp \left[ \sum_{1 \leq i < j \leq n} x_i x_j + \eps_{\lambda}(x_1,\ldots,x_n) \right] \, ,
	\end{align}with
	$$
	\eps_{\lambda}(x_1,\ldots,x_n) := \sum_{p=3}^{n} \lambda^{1-p/2} \sum_{1 \leq i_1 < \ldots < i_p \leq n} x_{i_1} \cdots x_{i_p} \ . 
	$$
	The terms corresponding to $[x_1^{m_1} \, \cdots \, x_n^{m_n}]$ in the expansion of $\exp \left[ \sum_{1 \leq i < j \leq n} x_i x_j + \eps_{\lambda}(x_1,\ldots,x_n) \right]$ 
	to which $\eps_{\lambda}(x_1,\ldots,x_n) $ contributes are in finite number (independently of $\lambda$) and are all of order $O(\lambda^{-1/2})$. Hence, taking $\lambda \to \infty$, the property \eqref{eq:theorem:eigendecomposition_mixed_products} is proved for $d=1$ in \eqref{eq:fl2}.
	\\

	\underline{\emph{Step 2: induction from $d$ to $d+1$}} We will now assume that all the stated properties of the decomposition of $L_d$ have been proven, and show that they hold true for $L_{d+1}$. Let us thus consider a pair of random trees in $\cX_{d+1}$ sampled from the correlated model given in Definition \ref{def:correlated_model}, with $N,N' \in \dN^{\cX_d}$ their corresponding vector representations. Given $k,N \in \dR^{\cX_d}$ we shall write $k \cdot N := \sum_{\beta \in \cX_d} k_{\beta} N_{\beta}$ (in all the occurences of this notation only a finite number of terms are non-vanishing, the sum is thus well-defined). The characteristic function of $\dPls_{d+1}$ is defined as $\hatdPls_{d+1}(k,k') := \dE\left[ e^{ik \cdot N + ik' \cdot N'}\right]$ and reads
	\begin{flalign}\label{eq:recursive_eigen_step1}
		& \hatdPls_{d+1}(k,k')  = \exp \left[ \lambda(1-s) \sum_{t \in \cX_d} \GWl_{d}(t)(e^{i k_t} + e^{i k'_t} - 2) + \lambda s \sum_{t,t' \in \cX_d} \dPls_{d}(t,t') (e^{i k_t + i k'_{t'}} -1) \right] \nonumber \\
		& = e^{\lambda \sum_{t \in \cX_d} \GWl_{d}(t)(e^{i k_t} -1) + \lambda \sum_{t \in \cX_d} \GWl_{d}(t)(e^{i k'_t} -1)} \exp \left[ \lambda s \sum_{t,t' \in \cX_d} \dPls_{d}(t,t') (e^{i k_t} -1)(e^{i k'_{t'}} -1) \right] \nonumber \\
		& = e^{\lambda \sum_{t \in \cX_d} \GWl_{d}(t)(e^{i k_t} -1) + \lambda \sum_{t \in \cX_d} \GWl_{d}(t)(e^{i k'_t} -1)}   \sum_{m \geq 0}s^m \frac{\lambda^m}{m!} \underbrace{\left( \sum_{t,t' \in \cX_d} \dPls_{d}(t,t') (e^{i k_t} -1)(e^{i k'_{t'}} -1)\right)^m}_{(i)} 
		\, . &&
	\end{flalign} Let us use the decomposition \eqref{eq:theorem:eigendecomposition} at step $d$ in $(i)$. Denoting $\gl_{d,\beta}(t) := \fl_{d,\beta}(t) \GWl_d(t)$, this gives 
	\begin{flalign*}
		(i) & =\left(\sum_{\beta \in \cX_d} s^{|\beta| -1} \left[ \sum_{t  \in \cX_d}\gl_{d,\beta}(t)(e^{i k_t} -1) \right] \left[ \sum_{t  \in \cX_d}\gl_{d,\beta}(t) (e^{i k'_{t}} -1)\right]\right)^m
		\, \\
		& = \sum_{\gamma = (\gamma_\beta)_{\beta \in \cX_d}} {m!} s^{- m + \sum_{\beta \in \cX_d}\gamma_{\beta}|\beta| } \\
		& \quad \quad \quad \times \prod_{\beta \in \cX_d} \frac{1}{\gamma_{\beta}!} \left[ \sum_{t  \in \cX_d}\gl_{d,\beta}(t)(e^{i k_t} -1) \right]^{\gamma_\beta} \left[ \sum_{t  \in \cX_d}\gl_{d,\beta}(t) (e^{i k'_{t}} -1)\right]^{\gamma_\beta} \one_{\sum_{\beta \in \cX_d}\gamma_{\beta} = m} \, ,
	\end{flalign*} 
	where we used a multinomial expansion.
	Summing $(i)$ for $m \geq 0$ gives an overall sum over all $\gamma = (\gamma_\beta)_{\beta \in \cX_d}$, that is over all $\cX_{d+1}$. Moreover, for $\gamma = (\gamma_\beta)_{\beta \in \cX_d} \in \cX_{d+1}$, one has $$|\gamma| = 1+ \sum_{\beta \in \cX_d}\gamma_{\beta}|\beta| \,. $$ 
	Hence, equation \eqref{eq:recursive_eigen_step1} becomes
	\begin{flalign*}
		\hatdPls_{d+1}(k,k') & = \sum_{\substack{\gamma \in \cX_{d+1} \\ \gamma = (\gamma_\beta)_{\beta \in \cX_d}}} s^{|\gamma|-1} \prod_{\beta \in \cX_d} \frac{1}{\gamma_{\beta}!} \left( \lambda \left[ \sum_{t  \in \cX_d}\gl_{d,\beta}(t)(e^{i k_t} -1) \right] \left[ \sum_{t  \in \cX_d}\gl_{d,\beta}(t) (e^{i k'_{t}} -1)\right]\right)^{\gamma_\beta} \nonumber \\
		& \quad \quad \quad \times  e^{\lambda \sum_{t \in \cX_d} \GWl_{d}(t)(e^{i k_t} -1) + \lambda \sum_{t \in \cX_d} \GWl_{d}(t)(e^{i k'_t} -1)} \nonumber \\
		& = \sum_{\substack{\gamma \in \cX_{d+1} \\ \gamma = (\gamma_\beta)_{\beta \in \cX_d}}} s^{|\gamma|-1} \hatgl_{d+1,\gamma}(k) \hatgl_{d+1,\gamma}(k')\,, 
	\end{flalign*} with 
	\begin{flalign*}
		\hatgl_{d+1,\gamma}(k) & := e^{\lambda \sum_{t \in \cX_d} \GWl_{d}(t)(e^{i k_t} -1)}  \prod_{\beta \in \cX_d}  \frac{1}{\sqrt{\gamma_{\beta}!}}  \left[ \sqrt{\lambda} \sum_{t  \in \cX_d}\gl_{d,\beta}(t)(e^{i k_t} -1)\right]^{\gamma_\beta} \nonumber \\
		& = e^{- \lambda} \sqrt{\prod_\beta \gamma_\beta!} \, [x^\gamma] e^{ \lambda \sum_{t \in \cX_d} \GWl_{d}(t) e^{i k_t} + \sum_{\beta \in \cX_d} x_{\beta} \sqrt{\lambda} \sum_{t  \in \cX_d}\gl_{d,\beta}(t)(e^{i k_t} -1)} \nonumber \\
		& = e^{- \lambda} \sqrt{\prod_\beta \gamma_\beta!} \, [x^\gamma] e^{ -  \sqrt{\lambda} \sum_{\beta,t \in \cX_d} x_{\beta} \gl_{d,\beta}(t) + \sum_{t} e^{i k_t} \left[\lambda \GWl_{d}(t) + \sum_{\beta} x_\beta \sqrt{\lambda} \gl_{d,\beta}(t) \right]} \, ,
	\end{flalign*} where $x = \set{x_\beta}_{\beta \in \cX_d}$ is a family of formal variables and $x^{\gamma}$ denotes $\prod_\beta x_{\beta}^{\gamma_\beta}$ when $\gamma = (\gamma_\beta)_{\beta \in \cX_d}$.  Recall that since the trees are finite, only a finite number of coordinates $\gamma_\beta$ are non-zero, which makes the infinite product problem disappear. The same arguments of normal convergence as in the case $d=1$ apply to justify the integral/sum permutations.
	
	As done in Step 1, we can invert the Fourier transform by integrating over every $k_t$, which gives  
	
	\begin{flalign*}
		\gl_{d+1,\gamma}(N) = e^{- \lambda} \sqrt{\prod_\beta \gamma_\beta!} \, [x^\gamma] e^{ -  \sqrt{\lambda} \sum_{\beta,t \in \cX_d} x_{\beta} \gl_{d,\beta}(t)} \prod_{t \in \cX_d} \frac{\left[\lambda \GWl_{d}(t) + \sum_{\beta} x_\beta \sqrt{\lambda} \gl_{d,\beta}(t) \right]^{N_t}}{N_t!} \, .
	\end{flalign*} 
	Dividing this expression by $\GWl_{d+1}(N)$ establishes that $L_{d+1}(N,N')$ satisfies the decomposition \eqref{eq:theorem:eigendecomposition} with $\fl_{d+1, \gamma}$ given by the following recursion
	\begin{equation}\label{eq:recursive_expression_f}
		\fl_{d+1,\gamma}(N) := \sqrt{\prod_{\beta \in \cX_d} \gamma_\beta!} \, [x^\gamma] e^{ -  \sqrt{\lambda} \sum_{\beta,t \in \cX_d} x_{\beta} \gl_{d,\beta}(t)} \prod_{t \in \cX_d} \left( 1 + \sum_{\beta \in \cX_d} \frac{x_\beta}{\sqrt{\lambda}}  \fl_{d,\beta}(t) \right)^{N_t} \, ,   
	\end{equation}
	which is independent of $s$.
	Taking $\gamma = \bigcdot$ in \eqref{eq:recursive_expression_f}, that is $\gamma_\beta = 0$ for all $\beta$, gives $\fl_{d+1,\bigcdot} = 1$ and proves condition \eqref{eq:theorem:eigendecomposition_trivial_tree} at depth $d+1$. \\
	
	\underline{\emph{Step 2.1: recursion for \eqref{eq:theorem:eigendecomposition_orthogonality} and \eqref{eq:theorem:eigendecomposition_mixed_products} at $d+1$}}
	Let us now prove the properties \eqref{eq:theorem:eigendecomposition_orthogonality} and \eqref{eq:theorem:eigendecomposition_mixed_products} at depth $d+1$.
	For any $\gamma^{(1)} = \{\gamma^{(1)}_\beta\}_{\beta \in \cX_{d}}$, $\ldots, \gamma^{(n)} = \{\gamma^{(n)}_\beta\}_{\beta \in \cX_{d}} \in \cX_{d+1}$, the recursion \eqref{eq:recursive_expression_f} gives
	
	\begin{multline*}
		\sum_{N \in \cX_{d+1}} \GWl_{d+1}(N) \fl_{d+1, \gamma^{(1)} }(N) \cdots \fl_{d+1, \gamma^{(n)} }(N) =  \sqrt{\prod_{i=1}^n \prod_{\beta \in \cX_d} \gamma^{(i)}_\beta!} \, \left[ \prod_{i=1}^{n} (x^{(i)})^{\gamma^{(i)}} \right] \\ \times \exp \left[ - \lambda - \sqrt{\lambda} \sum_{\beta,t  \in \cX_d} \sum_{i=1}^{n} x_{\beta}^{(i)} \gl_{d,\beta}(t) + \lambda \sum_{t \in \cX_d} \GWl_d(t) \prod_{i=1}^{n} \left(1 + \sum_{\beta \in \cX_d} \frac{x_{\beta}^{(i)}}{\sqrt{\lambda}} \fl_{d,\beta}(t) \right) \right] \, .
	\end{multline*}As in Step 1, when expanding the product in the exponential, the zero and first order terms simplify, which yields
	\begin{multline}\label{eq:recursive_mult_step1}
		\sum_{N \in \cX_{d+1}} \GWl_{d+1}(N) \fl_{d+1, \gamma^{(1)} }(N) \cdots \fl_{d+1, \gamma^{(n)} }(N) =  \sqrt{\prod_{i=1}^n \prod_{\beta \in \cX_d} \gamma^{(i)}_\beta!}  \\  \times \left[ \prod_{i=1}^{n} (x^{(i)})^{\gamma^{(i)}} \right] 
		\exp \left[\sum_{1 \leq i < j \leq n} \sum_{\beta,\beta' \in \cX_d} x_{\beta}^{(i)} x_{\beta'}^{(j)} \sum_{t \in \cX_d} \GWl_d(t) \fl_{d,\beta}(t) \fl_{d,\beta'}(t) + \eps_{\lambda}(x^{(1)}, \ldots, x^{(n)}) \right] \, ,
	\end{multline} with 
	$$
	\eps_{\lambda}(x^{(1)}, \ldots, x^{(n)}) := \sum_{p=3}^{n} \lambda^{1-p/2} \sum_{1 \leq i_1 < \ldots < i_p \leq n} \sum_{\beta_1, \ldots, \beta_p \in \cX_d} x^{(i_1)}_{\beta_1} \cdots x^{(i_p)}_{\beta_p} \sum_{t \in \cX_d} \GWl_d(t) \fl_{d,\beta_1}(t) \cdots \fl_{d,\beta_p}(t) \, . 
	$$
	Using the orthogonality property \eqref{eq:theorem:eigendecomposition_orthogonality} at step $d$, \eqref{eq:recursive_mult_step1} reads
	\begin{multline}\label{eq:recursive_mult_step2}
		\sum_{N \in \cX_{d+1}} \GWl_{d+1}(N) \fl_{d+1, \gamma^{(1)} }(N) \cdots \fl_{d+1, \gamma^{(n)} }(N) =  \sqrt{\prod_{i=1}^n \prod_{\beta \in \cX_d} \gamma^{(i)}_\beta!}   \\ \times \left[ \prod_{i=1}^{n} (x^{(i)})^{\gamma^{(i)}} \right] 
		e^{\sum_{1 \leq i < j \leq n} \sum_{\beta \in \cX_d} x_{\beta}^{(i)} x_{\beta}^{(j)}}
		\exp\left[\eps_{\lambda}(x^{(1)}, \ldots, x^{(n)}) \right] \, .
	\end{multline} 
	For $n=2$ the term $\eps_\lambda$ vanishes, hence
	\begin{align*}
		\sum_{N \in \cX_{d+1}} \GWl_{d+1}(N) \fl_{d+1, \gamma }(N) \fl_{d+1, \gamma' }(N) & = \prod_{\beta \in \cX_d} \left[ \sqrt{\gamma_\beta ! \gamma'_\beta! } [x^{\gamma_\beta} (x')^{\gamma'_\beta}] e^{x x'} \right] \\ & = \prod_{\beta \in \cX_d} \one_{\gamma_\beta = \gamma'_\beta} = \one_{\gamma = \gamma'} \, ,
	\end{align*} 
	establishing the orthogonality property \eqref{eq:theorem:eigendecomposition_orthogonality} for the rank $d+1$.
	
	Moreover, using the property \eqref{eq:theorem:eigendecomposition_mixed_products} at step $d$, $\sum_{t \in \cX_d} \GWl_d(t) \fl_{d,\beta_1}(t) \cdots \fl_{d,\beta_p}(t)$ 
	has a finite limit when $\lambda \to \infty$. Hence, as in Step $1$, the terms corresponding to $\left[ \prod_{i=1}^{n} (x^{(i)})^{\gamma^{(i)}} \right]$ in \eqref{eq:recursive_mult_step2} 
	to which $\eps_{\lambda}(x^{(1)}, \ldots, x^{(n)})$ contributes are in finite number (independent of $\lambda$) and are all of order $O(\lambda^{-1/2})$. Taking $\lambda \to \infty$ thus establishes property \eqref{eq:theorem:eigendecomposition_mixed_products} for the rank $d+1$.  Here again, previous computations are made rigorous since the trees are finite, by noticing that the series in \eqref{eq:recursive_expression_f} has infinite radius of convergence, and appealing to Fubini's theorem. We use the same arguments to make computations rigorous in the rest of the proof.\\
	
	\underline{\emph{Step 2.2: recursion for \eqref{eq:theorem:eigendecomposition_orthogonality_2} at $d+1$}}
	We will now prove the orthogonality relation \eqref{eq:theorem:eigendecomposition_orthogonality_2} for the rank $d+1$. To do so we introduce another family of formal variables $y=\{y_t\}_{t\in\cX_d}$ and rewrite \eqref{eq:recursive_expression_f} as
	\begin{align*}
		\fl_{d+1,\gamma}(N) & = \sqrt{\prod_{\beta \in \cX_d} \gamma_\beta!} \prod_{t \in \cX_d} N_t!  \, [x^\gamma y^N] e^{ -  \sqrt{\lambda} \sum_{\beta,t \in \cX_d} x_{\beta} \gl_{d,\beta}(t) + \sum_{t \in \cX_d} y_t + \sum_{\beta,t \in \cX_d}\frac{x_\beta y_t}{\sqrt{\lambda}}  \fl_{d,\beta}(t) } \\
		& = \frac{\prod_{t } N_t! }{\sqrt{\prod_\beta \gamma_\beta!}} \, [y^N] e^{\sum_{t \in \cX_d} y_t} \prod_{\beta \in \cX_d}\left( \sum_{t\in \cX_d} \fl_{d,\beta}(t) \left(\frac{y_t}{\sqrt{\lambda}} - \sqrt{\lambda} \, \GWl_d(t) \right)\right)^{\gamma_\beta}.
	\end{align*}
	The above expression gives that for all $N,N' \in \cX_{d+1}$,
	\begin{flalign*}
		& \sum_{\gamma \in \cX_{d+1}} \fl_{d+1, \gamma }(N)\fl_{d+1, \gamma }(N') = \prod_{t} N_t! N'_t! \, [x^N y^{N'}] e^{\sum_t (x_t + y_t)}\\
		& \quad \quad  \quad \quad \quad \quad \times e^{\sum_{\alpha,t,t'} \fl_{d,\alpha}(t) \fl_{d,\alpha}(t') \left(\frac{x_t}{\sqrt{\lambda}} - \sqrt{\lambda} \GWl_d(t) \right) \left(\frac{y_{t'}}{\sqrt{\lambda}} - \sqrt{\lambda} \GWl_d(t') \right)} \\
		& \quad \quad = \prod_{t} N_t! N'_t! \, [x^N y^{N'}] e^{\sum_t (x_t + y_t) + \sum_{t} \frac{1}{\GWl_d(t)}  \left(\frac{x_t}{\sqrt{\lambda}} - \sqrt{\lambda} \GWl_d(t) \right)\left(\frac{y_t}{\sqrt{\lambda}} - \sqrt{\lambda} \GWl_d(t) \right)}  \, ,
	\end{flalign*} where we used \eqref{eq:theorem:eigendecomposition_orthogonality_2} at step $d$ in the last step. This simplifies to 
	\begin{flalign*}
		\sum_{\gamma \in \cX_{d+1}} \fl_{d+1, \gamma}(N)\fl_{d+1, \gamma }(N') & = \prod_{t} N_t! N'_t! \, [x^N y^{N'}] \, e^{\lambda \GWl_{d}(t) + \sum_t \frac{x_t y_t}{\lambda \GWl_d(t)}} \\
		& = \prod_t \one_{N_t=N'_t} e^{\lambda \GWl_{d}(t)} (\lambda \GWl_{d}(t))^{-N_t} N_t! = \frac{\one_{N=N'}}{\GWl_{d+1}(N)} \, ,
	\end{flalign*} which proves \eqref{eq:theorem:eigendecomposition_orthogonality_2} at step $d+1$ and completes the proof of Theorem \ref{theorem:eigendecomposition}.

\end{proof}

\subsection{Computation of cyclic moments, proof of Theorem \ref{thm:negative_result}}
\label{section:proof_negative_result}

Theorem \ref{theorem:eigendecomposition} hereabove has a very natural Corollary that enables to compute the cyclic moments of the likelihood ratio.

\begin{corollary}[Cyclic moments]\label{corr:cyclic_moments}
	The \emph{$m-$th cyclic moment} of $L_d$ is defined as follows
	\begin{equation*}\label{eq:def:cyclic_moments}
		\Cls_{d,m} := \dEl_d\left[L_d(T_1,T_2)\cdots L_d(T_{m-1},T_m) L_d(T_m,T_1)\right],
	\end{equation*} where $T_1, \ldots, T_m$ are i.i.d. $\GWl_d$ in the above expectation. One has
	\begin{equation}\label{eq:corr:cyclic_moments}
		\Cls_{d,m} = \sum_{\beta \in \cX_d} (s^m)^{\card{\beta}-1} = \sum_{n \geq 1} A_{d,n}(s^m)^{n-1} = \Phi_d(s^m) ,
	\end{equation} where $A_{d,n}$, as defined in \eqref{eq:def:Adn}, denotes the number of unlabeled trees with $n$ vertices of depth at most $d$, and $\Phi_d$ is their generating function defined in Proposition \ref{prop:A_dn}. Note that in particular, the $\Cls_{d,m}$ do not depend on $\lambda$ (!) and by Proposition \ref{prop:A_dn} they are upper bounded for each $d$ and $s \in [0,1)$ by some constant $A=A(d,s)$. We thus denote $\Cs_{d,m} := \Cls_{d,m}$ in the sequel.
\end{corollary}

\begin{proof}[Proof of Corollary \ref{corr:cyclic_moments}]
	By Theorem \ref{theorem:eigendecomposition} we have 
	$$ L_d(t,t') = \sum_{\beta \in \cX_d} s^{\card{\beta} -1} \fl_{d,\beta}(t) \fl_{d,\beta}(t'),  $$
	hence, setting $\beta_{m+1} = \beta_1$,
	\begin{flalign*}
		\Cls_{d,m} & = \dEl_d\left[L_d(T_1,T_2)\cdots L_d(T_{m-1},T_m) L_d(T_m,T_1)\right] \\
		& = \sum_{\beta_1, \ldots, \beta_m \in \cX_d} s^{\sum_{i=1}^m (|\beta_i| -1)} \dEl_d \left[\prod_{i=1}^{m} \fl_{d,\beta_i}(T_i) \fl_{d,\beta_{i+1}}(T_i)\right] \\
		& = \sum_{\beta_1, \ldots, \beta_m \in \cX_d} s^{\sum_{i=1}^m (|\beta_i| -1)} \prod_{i=1}^{m} \dEl_d \left[\fl_{d,\beta_i}(T) \fl_{d,\beta_{i+1}}(T)\right] \\
		& = \sum_{\beta_1, \ldots, \beta_m \in \cX_d} s^{\sum_{i=1}^m (|\beta_i| -1)} \one_{\beta_1 = \, \ldots \, = \beta_m}\\
		& = \sum_{\beta \in \cX_d} (s^m)^{|\beta|-1} \ ,
	\end{flalign*}
	where we used the orthogonality property \eqref{eq:theorem:eigendecomposition_orthogonality} between the third and fourth line. All steps in the above computations are legitimate by Fubini's theorem, the integrability following from the Cauchy-Schwartz inequality since $$\dEl_d \left[\left|\fl_{d,\alpha}(T) \fl_{d,\alpha'}(T)\right|\right] \leq \dEl_d \left[(\fl_{d,\alpha}(T))^2\right]^{1/2} \dEl_d \left[ (\fl_{d,\alpha'}(T))^2\right]^{1/2} = 1 \, ,$$ by property \eqref{eq:theorem:eigendecomposition_orthogonality} of Theorem \ref{theorem:eigendecomposition}.
\end{proof}

We are now ready to give a proof of Theorem \ref{thm:negative_result}.

\begin{proof}[Proof of Theorem \ref{thm:negative_result}]
	According to Corollary \ref{corr:cyclic_moments}, one has 
	\begin{equation}\label{eq:proof:thm:negative_result}
		\dEl_d\left[ L_d (T,T')^2\right] = \Cs_{d,2} = \sum_{n \geq 1} A_{d,n} s^{2(n-1)}.
	\end{equation}
	Moreover, since $A_{d,n} \leq A_n$ (by Definition \ref{def:An_Adn}) and $A_n \underset{n \to \infty}{\sim} \frac{C}{n^{3/2}} \left( \frac{1}{\alpha} \right)^n$ by Proposition \ref{prop:otter}, the assumption $s \leq \sqrt{\alpha}$ ensures that $\dEl_d\left[ L_d (T,T')^2\right] = \sum_{n \geq 1} A_{d,n} s^{2(n-1)} \leq \sum_{n \geq 1} A_{n} s^{2(n-1)} < \infty$, uniformly in $d$.
	
	Then, applying Jensen's inequality yields
	\begin{equation*}
		\KL(\dPls_d \, || \, \dPl_d) = \dEls_d\left[ \log L_d (T,T')\right] \leq \log \dEls_d\left[ L_d (T,T')\right] \ .
	\end{equation*}
	By definition of the likelihood ratio, $\dEls_d\left[ f(T,T')\right] = \dEl_d\left[ L_d (T,T') f(T,T')\right]$ for every measurable function $f$, we can thus continue the above inequality as
	\begin{equation*}
		\KL(\dPls_d \, || \, \dPl_d) \leq \log \dEl_d\left[ L_d (T,T')^2\right] < \infty \ ,
	\end{equation*} uniformly in $d$, and conclude the proof.

\end{proof}



{As seen in this section, the boundedness (as $d \to \infty$) of $\dEl_d\left[ L_d (T,T')^2\right]-1$, the $\chi^2$ distance between $\dPls_d$ and $\dPl_d$, is enough to establish the impossibility result of Theorem \ref{thm:negative_result}. One could ask the reverse question: does the divergence of the $\chi^2$ distance imply the feasibility of one-sided detection? The answer to this question is negative; the divergence of the $\chi^2$ may occur even when $\KL_d$ does not diverge. For example, the $\chi^2$ diverges for $s > \sqrt{\alpha}$ and $\lambda>0$ arbitrarily small, including cases where $\lambda s<1$. We know that in this regime partial graph alignment is unfeasible, hence $\KL_d$ does not diverge. This is due to the fact that the mean value in the $\chi^2$ is dominated in this case by very rare events where the likelihood ratio $L_d$ is very large, which make it diverges while $\KL_d$, which is less sensitive to these atypical values because of its logarithmic dependency on $L_d$, remains bounded. In order to establish positive results through the divergence of (a lower bound on) $\KL_d$ we thus need a finer description of the random variable $L_d$ than the one provided by the $\chi^2$ distance, a goal that we will achieve in the high degree regime where some simplifications take place.}

\section{The high-degree regime: positive result when $s > \sqrt{\alpha}$ in the gaussian approximation}

In view of Definition \ref{def:correlated_model}, we recall that a pair of correlated trees $(t,t')$ of depth at most $d+1$ sampled from $\dPls_{d+1}$ are of the form $t = \set{N_\tau}_{\tau \in \cX_{d}}$ and $t' = \set{N'_\tau}_{\tau \in \cX_{d}}$ with 
\begin{equation}\label{eq:NN'M}
N_{\tau} := \Delta_\tau + \sum_{\tau' \in \cX_{d}} M_{\tau, \tau'} \quad \mbox{and} \quad
N'_{\tau'} := \Delta'_{\tau'} + \sum_{\tau \in \cX_{d}} M_{\tau, \tau'} \,.
\end{equation} with
\begin{equation*}
\Delta_{\tau}, \Delta'_{\tau} \sim\Poi(\lambda (1-s) \GWl_{d}(\tau)) \quad \mbox{and} \quad M_{\tau, \tau'} \sim \Poi(\lambda s \dPls_{d}(\tau,\tau')) \, ,
\end{equation*}
all these random variables being independent.

\subsection{Gaussian approximation in the high-degree regime}\label{subsection:gaussian_approx}

Let us define $y = (y_\beta)_{\beta \in \cX_d}$  and $y' = (y'_{\beta})_{\beta \in \cX_d}$ as follows:
\begin{flalign}
y_\beta := \frac{1}{\sqrt{\lambda}} \sum_{\tau \in \cX_{d}}\fl_{d,\beta}(\tau) (N_\tau - \lambda \GWl_{d}(\tau)) \label{eq:y_alpha} \\
y'_\beta := \frac{1}{\sqrt{\lambda}} \sum_{\tau \in \cX_{d}}\fl_{d,\beta}(\tau) (N'_\tau - \lambda \GWl_{d}(\tau)) \label{eq:y'_alpha} 
\end{flalign}
where the $\fl_{d,\beta}$ are defined in Theorem \ref{theorem:eigendecomposition}.
In other words, $y$ (resp. $y'$) is a centered version of $N$ (resp. $N'$), projected onto the basis of eigenvectors.

Let $(z,z') = ((z_\beta)_{\beta \in \cX_d},(z'_{\beta'})_{\beta' \in \cX_d}))$ be an (infinite-dimensional) centered Gaussian vector defined by its covariance matrix: 
\begin{equation}\label{eq:covariance_z}
\forall \beta, \beta' \in \cX_d, \quad \dE[z_{\beta}z_{\beta'}] = \dE[z'_{\beta}z'_{\beta'}] = \one_{\beta = \beta'}, \quad \dE[z_{\beta}z'_{\beta'}] = s^{|\beta|}\one_{\beta = \beta'}.
\end{equation}

Let us denote by $\dpls_{d+1}$ the joint distribution of $(y,y')$ when $(t,t')$ is drawn from $\dPls_{d+1}$, and $\gwl_{d+1}$ the marginal distribution of $y$ (or $y'$). In view of the orthogonality property \eqref{eq:theorem:eigendecomposition_orthogonality_2} in Theorem \ref{theorem:eigendecomposition}, the transformations $N \to y$ in \eqref{eq:y_alpha} and $N' \to y'$ in \eqref{eq:y'_alpha} are affine and bijective, and can be inverted as follows: 
\begin{equation*}
N_\tau = \lambda \, \GWl_{d}(\tau) + \sqrt{\lambda} \, \GWl_{d}(\tau) \sum_{\beta \in \cX_d} y_\beta \, \fl_{d,\beta}(\tau) \, .
\end{equation*} Hence,
\begin{equation}\label{eq:equality_of_KL_centered}
\KL(\dPls_{d+1} \| \GWl_{d+1} \otimes \GWl_{d+1}) = \KL(\dpls_{d+1} \| \gwl_{d+1} \otimes \gwl_{d+1}) \, .
\end{equation}


\begin{lemma}\label{lemma:cv_weak_dual}
When $\lambda \to \infty$, we have the following convergence in distribution:
\begin{equation}\label{eq:lemma:cv_weak_dual}
	(y,y') \overset{\mathrm{(d)}}{\longrightarrow} (z,z').
\end{equation}
\end{lemma}

\begin{proof}
Let us first precise the space in which $(y,y')$ and $(z,z')$ lie. This space is $\dR^{\cX_d} \times \dR^{\cX_d} $, which we endow with the following distance:
\begin{multline*}
	\rho_d((y^{(1)},y'^{(1)}),(y^{(2)},y'^{(2)})) := \\ \frac{1}{2}\sum_{\beta \in \cX_d} 2^{-\card{\beta}}\min(1,\left| y^{(1)}_\beta - y^{(2)}_\beta\right|) + \frac{1}{2}\sum_{\beta \in \cX_d} 2^{-\card{\beta}}\min(1,\left| y'^{(1)}_\beta - y'^{(2)}_\beta\right|) \, .
\end{multline*} 
The convergence of the sums is ensured by $\Phi_d(1/2) < \infty$, a consequence of Proposition \ref{prop:A_dn}. This metric turns $\dR^{\cX_d} \times \dR^{\cX_d} $ into a complete separable metric space (c.s.m.s. hereafter), and convergence in this metric is equivalent to simple convergence of each coordinate (see \cite{billingsley09}, Example 1.2, p.9). Importantly, convergence in distribution
in $\dR^{\cX_d} \times \dR^{\cX_d} $ amounts to convergence of all finite-dimensional distributions (see \cite{billingsley09}, Example 2.4, p.19). 

Let us then denote by $(k,k')$ a pair of real vectors in $\dR^{\cX_d} \times \dR^{\cX_d}$ with only a finite number of non-zero entries. We write $k \cdot y := \sum_{\beta \in \cX_d} k_{\beta} y_{\beta}$. We also define the following characteristic functions:
\begin{equation}
	\hatdpls(k,k') := \dE\left[ e^{ik \cdot y + i k' \cdot y'} \right] \quad \mbox{ and } \quad
	\hatrs(k,k') := \dE\left[ e^{ik \cdot z + i k' \cdot z'} \right] \, .
\end{equation}
Proving Lemma \ref{lemma:cv_weak_dual} thus amounts to showing the simple convergence $\hatdpls(k,k') \to \hatrs(k,k')$ when $\lambda \to \infty$. Since the (Gaussian) limit distribution is entirely determined by its moments, it suffices to show the convergence of the cumulants (for a proof, see Theorem 4.5.5 in~\cite{Chung}). 

The covariance structure of $(z,z')$ given in \eqref{eq:covariance_z} immediately yields
\begin{equation}\label{eq:cumulant_zz'}
	\hatrs(k,k') = \exp \left[-\frac{1}{2} \sum_{\beta \in \cX_d} ((k_\beta)^2 + (k'_\beta)^2 + 2 s^{|\beta|}k_\beta k'_\beta) \right] \, .
\end{equation}



In view of \eqref{eq:NN'M}, \eqref{eq:y_alpha} and \eqref{eq:y'_alpha}, writing $\fl_{d}(\tau) := (\fl_{d,\beta}(\tau))_{\beta \in \cX_d}$, one has
\begin{multline*}
	e^{ik \cdot y + i k' \cdot y'} = \exp\left[ - \sqrt{\lambda} \sum_{\tau \in \cX_d} \GWl_d(\tau) (i k \cdot \fl_{d}(\tau) + ik'  \cdot \fl_{d}(\tau)) \right] \\
	\times \prod_{\tau, \tau' \in \cX_d} \left( \exp \left[\frac{1}{\sqrt{\lambda}}  (ik \cdot \fl_{d}(\tau) + i k' \cdot \fl_{d}(\tau')) \right]\right)^{M_{\tau,\tau'}} \times \prod_{\tau \in \cX_d} \left( \exp \left[\frac{1}{\sqrt{\lambda}}  ik \cdot \fl_{d}(\tau) \right]\right)^{\Delta_{\tau}} \\
	\times \prod_{\tau \in \cX_d} \left( \exp \left[\frac{1}{\sqrt{\lambda}}  ik' \cdot \fl_{d}(\tau) \right]\right)^{\Delta'_{\tau}}.
\end{multline*} The variables $M_{\tau,\tau'}, \Delta_\tau, \Delta'_\tau$ being independent Poisson variables, taking the expectation gives
\begin{multline*}
	\hatdpls(k,k')  = \exp\left[ - \sqrt{\lambda} \sum_{\tau \in \cX_d} \GWl_d(\tau) (i k \cdot \fl_{d}(\tau) + ik'  \cdot \fl_{d}(\tau)) \right] \\
	\times \exp\left[ \lambda(1-s) \sum_{\tau \in \cX_d} \GWl_d(\tau) \left( e^{\frac{1}{\sqrt{\lambda}} i k \cdot \fl_{d}(\tau)} + e^{\frac{1}{\sqrt{\lambda}} i k' \cdot \fl_{d}(\tau)} - 2 \right) \right]\\
	\times \exp\left[ \lambda s \sum_{\tau, \tau' \in \cX_d} \dPls_{d}(\tau,\tau') \left( e^{\frac{1}{\sqrt{\lambda}} (i k \cdot \fl_{d}(\tau)+ i k' \cdot \fl_{d}(\tau'))} - 1 \right) \right] \, .
\end{multline*}
The cumulants of $(y,y')$ are obtained by expanding the logarithm of the last expression in power series in $k,k'$. Using that $\sum_{\tau' \in \cX_d} \dPls_{d}(\tau,\tau') = \GWl_d(\tau)$, the first-order (linear) terms compensate to $0$, which translates the fact that $\dE[y_\alpha]=\dE[y'_\alpha]=0$. The second-order terms in $\log \hatdpls(k,k')$ evaluate to 
\begin{flalign*}
	& - \lambda (1-s)  \sum_{\tau \in \cX_d} \GWl(\tau) \frac{1}{2 \lambda} \sum_{\beta,\beta' \in \cX_d} \fl_{d,\beta}(\tau) \fl_{d,\beta'}(\tau) \left( k_\beta k_{\beta'} + k'_\beta k'_{\beta'} \right)\\
	& - \lambda s \sum_{\tau, \tau' \in \cX_d}  \dPls_{d}(\tau,\tau')  \\
	& \quad \quad \quad \quad \times \frac{1}{2 \lambda} \sum_{\beta,\beta' \in \cX_d} \left( \fl_{d,\beta}(\tau) \fl_{d,\beta'}(\tau) k_\beta k_{\beta'} + \fl_{d,\beta}(\tau') \fl_{d,\beta'}(\tau') k'_\beta k'_{\beta'} + 2 \fl_{d,\beta}(\tau) \fl_{d,\beta'}(\tau') k_\beta k'_{\beta'} \right)
	\, .
\end{flalign*} Using the orthogonality property \eqref{eq:theorem:eigendecomposition_orthogonality} of the eigenvectors in Theorem \ref{theorem:eigendecomposition}, 
the previous equation simplifies into
\begin{equation*}
	- \frac{1}{2} \sum_{\beta \in \cX_d} \left( (k_\beta)^2 + (k'_\beta)^2 \right) 
	- s \sum_{\tau, \tau' \in \cX_d} \dPls_{d}(\tau,\tau') \sum_{\beta,\beta' \in \cX_d} \fl_{d,\beta}(\tau) \fl_{d,\beta'}(\tau') k_\beta k'_{\beta'} 
	\, ,
\end{equation*} which in turn reads, using $\dPls_{d}(\tau,\tau') = \GWl_{d}(\tau) \GWl_{d}(\tau') \sum_{\gamma \in \cX_d} s^{|\gamma| -1} \fl_{d,\gamma}(\tau) \fl_{d,\gamma}(\tau')$:
\begin{flalign*}
	& - \frac{1}{2} \sum_{\beta \in \cX_d} \left( (k_\beta)^2 + (k'_\beta)^2 \right) 
	- s  \sum_{\beta,\beta',\gamma \in \cX_d} s^{|\gamma| -1} k_\beta k'_{\beta'} \\
	& \quad \quad \quad \quad \quad \times \left( \sum_{\tau \in \cX_d} \GWl_{d}(\tau) \fl_{d,\beta}(\tau) \fl_{d,\gamma}(\tau)  \right) \left( \sum_{\tau' \in \cX_d} \GWl_{d}(\tau') \fl_{d,\beta'}(\tau') \fl_{d,\gamma}(\tau')  \right)\\
	&  \quad \quad \quad \quad \quad \quad \quad \quad \quad  = - \frac{1}{2} \sum_{\beta \in \cX_d} \left( (k_\beta)^2 + (k'_\beta)^2 + 2 s^{|\beta|} k_\beta k'_{\beta} \right) 
	\, ,
\end{flalign*} which is exactly the second cumulant of $(z,z')$ in \eqref{eq:cumulant_zz'}. The remaining step is to show that the higher order cumulants tend to $0$ when $\lambda$ gets large. The terms of order $n$ in the expansion of $\log \hatdpls(k,k')$ in powers of $(k_\beta,k'_\beta)$ depends explicitly on $\lambda$ through the prefactor $\lambda^{1-n/2}$; moreover their implicit dependency through the eigenvectors $\fl_{d,\beta}$ is of the form 
$$ \sum_{t \in \cX_d} \GWl_d(t) \fl_{d,\beta_1}(t) \cdots \fl_{d,\beta_p}(t) \ . $$ These quantities have been proved to remain finite when $\lambda \to \infty$ by property \eqref{eq:theorem:eigendecomposition_mixed_products} of Theorem \ref{theorem:eigendecomposition}. This shows that all the cumulants of $(y,y')$ of order $\geq 3$ tend to $0$ when $\lambda \to \infty$, and hence establishes the desired convergence in distribution.
\end{proof}

\begin{remark}
A hint of the Gaussian convergence in the limit $\lambda \to \infty$ can be a posteriori read from Equation \eqref{eq:theorem:eigendecomposition_mixed_products}.
To explain this point let us first define the Hermite polynomials
\begin{equation*}
	H_m(y) := \sqrt{m!} [x^m] e^{-\frac{1}{2}x^2 - xy} \, ,
\end{equation*} 
that are orthogonal with respect to the Gaussian distribution, i.e. $\dE[H_m(X) H_{m'}(X)] = \one_{m=m'}$ if $X$ is a standard Gaussian random variable. The average of a product of $n$ orthogonal polynomials is studied in the general theory of orthogonal polynomials under the name of linearization coefficient (see e.g. \cite{linearization_op}). For Hermite polynomials one finds
\begin{equation*}
	\dE[H_{m_1}(X) \cdots H_{m_n}(X)] 
	= \sqrt{\prod_{i=1}^{n} m_i!} \, [x_1^{m_1} \, \cdots \, x_n^{m_n}] \exp \left[ \sum_{1 \leq i < j \leq n} x_i x_j \right] \, ,
\end{equation*}
to be compared with the expression \eqref{eq:fl2} in the Charlier case. The right hand side of \eqref{eq:theorem:eigendecomposition_mixed_products} can thus be written as
\begin{equation*}
	\dE\left[\prod_{i=1}^n \prod_{\gamma \in \cX_{d-1}} H_{\beta^{(i)}_\gamma}(y_\gamma) \right] \ ,
\end{equation*}
with $y_\gamma$ i.i.d. standard Gaussian random variables. Comparing this expression with the left hand side of \eqref{eq:theorem:eigendecomposition_mixed_products} suggests that in the limit $\lambda \to \infty$ the eigenvectors $\fl_{d,\beta}$ behave as products of Hermite polynomials, and thus that the measure with respect to which they are orthogonal becomes the product measure of independent Gaussians.
\end{remark}

\subsection{Kullback-Leibler divergence in the high-degree regime}

The following result compares the $\KL-$divergence with finite $\lambda$ to the $\KL-$divergence between the limiting Gaussian distributions of Lemma \ref{lemma:cv_weak_dual}. Together with the weak convergence established in Lemma \ref{lemma:cv_weak_dual}, it will be instrumental to the proof of Theorem \ref{thm:positive_result}. 

\begin{proposition}\label{prop:gaussian_KL_ineq}
Denoting
\begin{equation}\label{eq:theorem:def_KL}
	\KLls_d := \KL(\dPls_d || \dPl_d) \, ,
\end{equation} one has the following:
\begin{equation}\label{eq:prop:gaussian_KL_limitinf} 
	\forall d \geq 1, \; \liminf_{\lambda \to\infty} \KLls_d \geq \KLs_d := - \frac{1}{2}\sum_{\beta \in \cX_{d-1}} \log(1-s^{2 |\beta|}) = \frac{1}{2} \log\Cs_{d,2} \, ,
\end{equation}  where $\Cs_{d,2}$ was defined in Lemma~\ref{corr:cyclic_moments},  and shown  by Equation~\eqref{eq:proof:thm:negative_result} to be equal to
\begin{equation}\label{eq:rappel:KLs}
	\Cs_{d,2} = \dEls_d[L_d] = \sum_{n \geq 1} A_{d,n}(s^2)^{n-1}  \, .
\end{equation}
\end{proposition}

\begin{proof}[Proof of Proposition \ref{prop:gaussian_KL_ineq}]
Fix $d \geq 1$. In \eqref{eq:equality_of_KL_centered}, we established that $\KLls_d$ is also the $\KL$-divergence $\KL(\dpls_{d} \| \gwl_{d} \otimes \gwl_{d})$ where $\dpls_{d}$ is the distribution of $(y,y')$ defined in Section \ref{subsection:gaussian_approx}. Moreover, Lemma \ref{lemma:cv_weak_dual} establishes that $(y,y')$ converges in distribution to a centered gaussian vector $(z,z')$ defined by its covariance matrix:
\begin{equation}\label{eq:covariance_z_bis}
	\forall \beta, \beta' \in \cX_{d-1}, \quad \dE[z_{\beta}z_{\beta'}] = \dE[z'_{\beta}z'_{\beta'}] = \one_{\beta = \beta'}, \quad \dE[z_{\beta}z'_{\beta'}] = s^{|\beta|}\one_{\beta = \beta'}.
\end{equation}
If we denote by $p^{(s)}_1$ the joint distribution of the gaussian vector $(z,z')$ and $p^{(s)}_0$ the product of the marginals, the KL-divergence $\KL(p^{(s)}_1 || p^{(s)}_0)$ is easily given by $-\frac{1}{2}\log\det\Sigma$, where $\Sigma$ is the covariance matrix of $(z,z')$, which is similar to a matrix with diagonal blocks of the form $\begin{pmatrix} 1 & s^{|\beta|} \\ s^{|\beta|} & 1\end{pmatrix}$ for all $\beta \in \cX_{d-1}$, which gives 
\begin{flalign*}
	\KL(p^{(s)}_1 || p^{(s)}_0) = - \frac{1}{2} \log \prod_{\beta \in \cX_{d-1}}  (1-s^{2|\beta|}) \, .
\end{flalign*} The last term is indeed $\KLs_d$ as defined in \eqref{eq:prop:gaussian_KL_limitinf}, since
\begin{flalign*}
	\dEls_d[L_d]  = \sum_{\beta \in\cX_{d}} s^{2 (|\beta|-1)}  = \prod_{\beta \in \cX_{d-1}} \sum_{\gamma_\beta \geq 0} s^{2 {\gamma_\beta} |\beta|} = \prod_{\beta \in \cX_{d-1}}  \frac{1}{1-s^{2|\beta|}} \, .
\end{flalign*}

The proof is concluded by appealing to the lower semi-continuity property of the $\KL$-divergence with respect to the weak convergence of its arguments (see e.g. \cite{PolyanskiyLecturenotes}, Theorem 3.6), and to Lemma \ref{lemma:cv_weak_dual}, yielding
\begin{equation*}
	\liminf_{\lambda \to\infty} \KLls_d \geq \KLs_d \, .
\end{equation*}
\end{proof}

{The inequality provided by Proposition \ref{prop:gaussian_KL_ineq} is sufficient to show the positive result of Theorem \ref{thm:positive_result}, which is done in the following section. It is however natural to wonder whether the $\KL-$divergence with finite $\lambda$ actually converges to the $\KL-$divergence between the limiting Gaussian distributions. This is indeed true, and we refer the interested reader to Appendix \ref{appendix:convergence_KL} for a proof.}

\subsection{Propagating bounds on the $\KL-$divergence, proof of Theorem \ref{thm:positive_result}}
The goal of this section is to use the result of Proposition \ref{prop:gaussian_KL_ineq} and the fact that in view of \eqref{eq:rappel:KLs} for $s>\sqrt{\alpha}$ (where $\alpha$ is Otter's constant), $\KLs_d \to +\infty$ with $d$, in order to obtain Theorem \ref{thm:positive_result}, that is that for fixed $s>\sqrt{\alpha}$, there exists $\lambda(s)$ such that one-sided detection is feasible for $\lambda \geq \lambda(s)$.

The following Lemma shows that if $s > \sqrt{\alpha}$, for any small (resp. any large but bounded) probability that we fix, there exists $\lambda_1>0$, a depth $d_0$ and an event $S$ that has this small (resp. large) probability under $\dPl_{d_0}$ (resp. $\dPls_{d_0}$) for $\lambda \geq \lambda_1$. The proof is deferred to Appendix \ref{appendix:proof:lemma:initialize_lower_bound_on_KL}.


\begin{lemma}\label{lemma:initialize_lower_bound_on_KL}
Assume that $s> \sqrt{\alpha}$. Then for any $c\in (0,2/15)$ and any $\eps\in (0,1)$, there exists $\lambda_1=\lambda_1(s,c,\eps)>0$ and $d_0 = d_0(s,c,\eps)\in\dN$ such that, for all $\lambda\geq \lambda_1$, there exists an event $S =S(s,c,\eps) \subset \cX_{d_0}^2$ for which the following inequalities hold:
\begin{equation*}
	\dPls_{d_0}( S)\geq c \quad \mbox{and} \quad \dPl_{d_0}(S)\leq \eps \, .
\end{equation*}
\end{lemma}

Now that we know that this event $S$ exists at a certain initial depth $d_0$, we want to propagate similar bounds for arbitrary depth $d \geq d_0$. This is the object of the following Proposition, proved in Appendix \ref{appendix:proof:prop:bounding_LR_by_induction}.

\begin{proposition}\label{prop:bounding_LR_by_induction}
For any fixed  $c\in(0,1)$ there exist constants $\eps=\eps(s,c) \in (0,1)$ and $\lambda_0 = \lambda_0(s,c)>0$ such that the following holds. For any $\lambda\geq \lambda_0$, any $d\in\dN$, if there exists an event $S\subset \cX_d^2$ such that
\begin{equation*}
	\dPl_d(S)\leq \eps \quad \mbox{and} \quad \dPls_d(S)\geq c \, ,
\end{equation*} then there exists an event $S'\subset \cX_{d+1}^2$ such that 
\begin{equation*}
	\dPl_{d+1}(S')\leq \frac{1}{2}\dPl_{d}(S) \leq \frac{\eps}{2}  \quad \mbox{and} \quad \dPls_{d+1}(S')\geq c \, .
\end{equation*}

In fact, with the usual notations $t=\{N_\tau\}_{\tau\in\cX_d}$, $t'=\{N'_\tau\}_{\tau\in\cX_d}$ for elements of $\cX_{d+1}$, and denoting, for all $\tau\in \cX_d$
\begin{equation*}
	\widetilde{N}_\tau:=N_\tau -\lambda \GWl_d(\tau) \quad \mbox{and} \quad \widetilde{N}'_\tau= N'_\tau-\lambda \GWl_d(\tau)\, ,
\end{equation*}
the event $S'$ in the above is defined from $S$ in the following way :
\begin{equation}\label{eq:def_Z}
	S'=\set{Z_S \geq \sigma} \, , \ \text{where} \qquad
	Z_S :=\sum_{(\tau,\tau')\in S}\widetilde{N}_\tau \widetilde{N}'_{\tau'} \, ,
\end{equation} for some suitable threshold $\sigma = \sigma(S)$.
\end{proposition}

Together, Lemma \ref{lemma:initialize_lower_bound_on_KL} and Proposition \ref{prop:bounding_LR_by_induction} yield the proof of Theorem \ref{thm:positive_result}.

\begin{proof}[Proof of Theorem \ref{thm:positive_result}]
Assume that $s > \sqrt{\alpha}$.
Choose $c\in(0,2/15)$ and let $\eps=\eps(s,c)$, $\lambda_0=\lambda_0(s,c)$ be the corresponding quantities from Proposition \ref{prop:bounding_LR_by_induction}. Now that $c,\eps$ are fixed, we appeal to  Lemma \ref{lemma:initialize_lower_bound_on_KL} to obtain some $\lambda_1=\lambda_1(s,c,\eps)$ and $d_0 = d_0(s,c,\eps) \in\dN$ such that, taking $\lambda\geq \lambda_0 \vee \lambda_1$, there exists some event $S_{d_0}\subset\cX_{d_0}^2$ such that
\begin{equation*}
	\dPl_{d_0}(S)\leq \eps \quad \mbox{and} \quad \dPls_{d_0}(S)\geq c \, .
\end{equation*} Proposition \ref{prop:bounding_LR_by_induction} then ensures the existence of a sequence of events $S_d\subset \cX_d^2$, $d>d_0$ such that
\begin{equation*}
	\dPl_d(S_d)\leq 2^{-(d-d_0)}\eps \quad \mbox{and} \quad \dPls_d(S_d) \geq c \, ,
\end{equation*}
which shows that the test $\cT_d = \one_{S_d}$ achieves one-sided detection.

\end{proof}

\section*{Acknowledgments}
The authors would like to thank Marc Lelarge and Guillaume Barraquand for useful discussions. 
The first author was supported by the French government under management of Agence Nationale de la Recherche as part of the “Investissements d’avenir” program, reference ANR19-P3IA-0001 (PRAIRIE 3IA Institute).

\bibliographystyle{plain}

\bibliography{biblio}

\begin{appendix}
\section{Convergence of $\KL_d$ in the high degree regime}\label{appendix:convergence_KL}

{As mentioned earlier, the inequality provided in Proposition \ref{prop:gaussian_KL_ineq} is sufficient to show the positive result of Theorem \ref{thm:positive_result}. However, we are able to prove a stronger version of Proposition \ref{prop:gaussian_KL_ineq}, in line with the intuition: the likelihood-ratio converges weakly under $\dPl_d$ when $\lambda \to \infty$, and the $\KL-$divergence with finite $\lambda$ converges to the $\KL-$divergence between the limiting Gaussian distributions of Lemma \ref{lemma:cv_weak_dual}. To show this convergence, we prove the following result, established in the more general setting of complete separable metric spaces.}

\begin{proposition}\label{proposition:KL_convergence}
Let $\cX$ be a complete separable metric space (c.s.m.s) endowed with its Borel $\sigma$-field. Let for all $\lambda>0, p_{0,\lambda}, p_{1,\lambda}$ be two probability measures on $\cX$ such that $p_{1,\lambda}\ll p_{0,\lambda}$, with $\dE_{0,\lambda}[\cdot]$ and $\dE_{1,\lambda}[\cdot]$ the corresponding expectations. Denote then by $L^{(\lambda)}$ the likelihood ratio $\frac{\mathrm{d}p_{1,\lambda}}{\mathrm{d}p_{0,\lambda}}$. Assume that as $\lambda\to\infty$, one has the weak convergences 
\begin{equation}\label{eq:proposition:KL_convergence:cv_distribution}
	p_{0,\lambda} \overset{w}{\longrightarrow} p_0\hbox{  and }p_{1,\lambda}\overset{w}{\longrightarrow}p_1 \, .
\end{equation}
Assume also that for some finite constant $c>0$ one has
\begin{equation}\label{eq:proposition:KL_convergence:second_moment_constant}
	\forall \lambda>0,\; \dE_{0,\lambda}[ (L^{(\lambda)})^2]\leq c \, .
\end{equation} Then $p_1 \ll p_0$, and if we assume further that the likelihood ratio $L := \frac{\mathrm{d}p_1}{\mathrm{d}p_0}$ satisfies
\begin{equation}\label{eq:proposition:KL_convergence:second_moment_cv}
	\dE_{0,\lambda}[ (L^{(\lambda)})^2] \underset{\lambda \to \infty}{\longrightarrow} \dE_{0} [ L^2 ] \, ,
\end{equation} it then follows that:
\begin{itemize}
	\item[$(i)$] The distribution of $L^{(\lambda)}$ under $p_{0,\lambda}$ converges weakly as $\lambda\to\infty$ to the distribution of $L$ under $p_0$;
	\item[$(ii)$] \begin{equation}\label{eq:proposition:KL_convergence:results}
		\KL(p_{1,\lambda} || p_{0,\lambda}) \underset{\lambda \to \infty}{\longrightarrow} \KL(p_1 || p_0) \, .
	\end{equation}
\end{itemize}
\end{proposition}

\begin{remark}
Note that the assumption \eqref{eq:proposition:KL_convergence:second_moment_cv} in the previous Proposition is crucial. To see this, consider the following counter-example. Set $p_{0,\lambda}=p_0$ to be the uniform distribution on $[0,1]$. Then, for $\lambda \geq 1$, we subdivide the interval $I=[0,1]$ in $n = 2 \lfloor \lambda \rfloor$ intervals of the form $\left( \frac{i}{n}, \frac{i+1}{n}\right)$, and define $p_{1,\lambda}$ to be absolutely continuous with respect to $p_0$, with a density $1/2$ (resp. $3/2$) on the interval $\left( \frac{i}{n}, \frac{i+1}{n}\right)$ when $i$ is even (resp. $i$ is odd). In this setting, the likelihood ratio is bounded, $p_{1,\lambda}$ converges weakly to $p_1=p_0$, hence $\KL(p_1 || p_0)=0$, but for all finite $\lambda \geq 1$, 
$$ \KL(p_{1,\lambda} || p_{0,\lambda}) = \frac{1}{4} \log(1/2) + \frac{3}{4} \log(3/2)$$ 
which is strictly positive, so that \eqref{eq:proposition:KL_convergence:results} does not hold. 
\end{remark}

\begin{proof}[Proof of Proposition \ref{proposition:KL_convergence}]
	Recall that $L^{(\lambda)}$ denotes the likelihood ratio $\frac{\mathrm{d}p_{1,\lambda}}{\mathrm{d}p_{0,\lambda}}$.
	
	\underline{\textit{Proof of the fact that $p_1 \ll p_0$}}
	Let us consider the family of measures $\set{\cQ^{(\lambda)}}_\lambda$ where $\cQ^{(\lambda)}$ is the pushforward of $p_{0,\lambda}$ on $\cX \times \dR_+$, endowed with the product measure/topology, by the mapping $\omega \mapsto (\omega,L^{(\lambda)}(\omega))$. In the sequel we denote $\hat{\cF}$ the product of the Borel $\sigma$-algebra on $\cX$ with the trivial sigma-algebra on $\dR_+$.
	
	Thanks to the weak convergence of $p_{0,\lambda}$ and to uniform integrability of $\{L^{(\lambda)}\}_\lambda$, which follows from assumption \eqref{eq:proposition:KL_convergence:second_moment_constant}, the two families of marginals are tight and hence the family $\set{\cQ^{(\lambda)}}_\lambda$ is tight.
	
	Since $\cX$ is a c.s.m.s., so is $\cX\times \dR_+$ and Prokhorov's theorem applies: for any sequence $\lambda_k$ diverging to $+\infty$, one can extract a subsequence $(\lambda_{\phi(k)})_k$ along which  $\cQ^{(\lambda_{\phi(k)})}$ converges weakly to a limit. Let us denote such a weak limit by $\cQ$, a distribution on $\cX \times \dR_+$. Since by assumption $p_{0,\lambda} \overset{w}{\longrightarrow} p_0$, the marginal of $\cQ$ on $\cX$ is necessarily $p_0$.
	
	By uniform integrability of the variables $\set{L^{(\lambda)}(\omega)}_\lambda$, for any bounded continuous function $h: \cX \to \dR$, the variables $\set{h(\omega) L^{(\lambda_{\phi(k)})}(\omega)}_k$ are also uniformly integrable, and
	
	\begin{equation}\label{eq:limit_hL_first}
		\dE_{0,\lambda_{\phi(k)}}[h(\omega) L^{(\lambda_{\phi(k)})}(\omega)] \underset{k \to \infty}{\longrightarrow}  \dE_{(\omega,z)\sim\cQ}[h(\omega)z] \, .
	\end{equation}
	On the other hand 
	\begin{equation}\label{eq:limit_hL_second}
		\dE_{0,\lambda_{\phi(k)}}[h(\omega) L^{(\lambda_{\phi(k)})}(\omega)] = \dE_{1,\lambda_{\phi(k)}}[h(\omega)] \underset{k \to \infty}{\longrightarrow} \dE_{1}[h(\omega)] \, .
	\end{equation}  
	Introducing the variable $\hat{z}(\omega) := \dE_{(\omega,z)\sim \cQ}[z \, | \, \hat{\cF}]$, the r.h.s. of \eqref{eq:limit_hL_first} also reads 
	$$
	\dE_{(\omega,z)\sim\cQ}[h(\omega)z] = \dE_{(\omega,z)\sim \cQ}[h(\omega) \dE_{(\omega,z)\sim \cQ} [z \, | \, \hat{\cF}] ] = \dE_{(\omega,z)\sim \cQ}[h(\omega)\hat{z}(\omega)]=\dE_{\omega\sim p_0}[h(\omega)\hat{z}(\omega)] \, .
	$$
	Hence, identifying limits in \eqref{eq:limit_hL_first} and \eqref{eq:limit_hL_second} gives that $p_1 \ll p_0$ and that $\hat{z}$ is $p_0-$almost surely the Radon-Nikodym derivative $\frac{\mathrm{d}p_1}{\mathrm{d}p_0}$, which is essentially unique, thus does not depend on the particular choice of subsequence $(\lambda_{\phi(k)})_k$ and weak limit $\cQ$. In line with the statement of the Proposition, we shall denote $L := \frac{\mathrm{d}p_1}{\mathrm{d}p_0}$ in the sequel.
	
	\underline{\textit{Proof of $(i)$}} We are now ready to prove $(i)$, namely that the distribution of $L^{(\lambda)}$ under $p_{0,\lambda}$ converges weakly as $\lambda\to\infty$ to the distribution of $L$ under $p_0$. As previously, consider a subsequence 
	$(\lambda_{\phi'(k)})_k$ along which $\cQ^{(\lambda_{\phi'(k)})}$ converges weakly to $\cQ'$. By the same arguments as before, one has
	\begin{equation}\label{eq:limit_hL_third}
		\dE_{0,\lambda_{\phi'(k)}}[h(\omega) L^{(\lambda_{\phi'(k)})}(\omega)] \underset{k \to \infty}{\longrightarrow}  \dE_{(\omega,z')\sim\cQ'}[h(\omega)z'] \, ,\end{equation}
	and $\dE_{\cQ'}[ z' \, | \, \hat{\cF}] = L$. By Fatou's Lemma,
	\begin{equation}\label{eq:moment_carre_un}
		\liminf_k \dE_{0,\lambda_{\phi'(k)}}[(L^{(\lambda_{\phi'(k)})})^2] = 2 \liminf_k \int_{0}^{+ \infty} x p_{0,\lambda_{\phi'(k)}}(L^{(\lambda_{\phi'(k)})} \geq x) \, \mathrm{d}x \geq \dE_{(\omega,z')\sim\cQ'}[(z')^2] \, .
	\end{equation}
	
	Since $L$ is $\hat{\cF}-$measurable then in turn $\dE_{\cQ'}[ z' \, | \, L ] = L$.  By the conditional variance fomula,
	$$
	\Var_{\cQ'}[z'] = \dE_{\cQ'}[\Var_{\cQ'}[z'\, | \, L ]]+\Var_{\cQ'}[L] 
	$$
	so that 
	\begin{equation}\label{eq:moment_carre_deux}
		\Var_{\cQ'}[z']\ge \Var_{\cQ'}[L] 
	\end{equation} 
	with equality if and only if $z'$ is $\hat{\cF}-$measurable, that is if $z'=L$, $\cQ'$-almost surely. 
	
	On the other hand, by \eqref{eq:moment_carre_un} and assumption  \eqref{eq:proposition:KL_convergence:second_moment_cv} one also has
	$$ \dE_{\cQ'}[(z')^2] \leq \liminf_k \dE[(L^{(\lambda_{\phi'(k)})})^2] = \dE[L^2] \, .  $$
	Since $L$ is a likelihood ratio one has $\dE_\cQ[L]=1$, and taking $h =1$ in \eqref{eq:limit_hL_third} gives $\dE_{\cQ'}[z']=1$. Hence inequality \eqref{eq:moment_carre_deux} is an equality and we have $z'=L$, $\cQ'$-almost surely. There is only one possible distribution for $z'$ which is that of the likelihood ratio $L=\frac{dp_1}{dp_0}$ under $p_0$, hence weak convergence follows.
	
	\underline{\textit{Proof of $(ii)$}} To establish point $(ii)$, we must show that $\dE_{0,\lambda} [\phi(L^{(\lambda)})] \underset{\lambda \to \infty}{\longrightarrow} \dE_0 [\phi(L)]$ where $\phi(x):=x\log(x)$. We already have from point $(i)$ that the distribution of $\phi(L^{(\lambda)})$ weakly converges to the distribution of $\phi(L)$. The conclusion will follow if we can show that the family of random variables $\phi(L^{(\lambda)})$ under $p_{0,\lambda}$ is uniformly integrable. This follows in turn if we can establish that for some $\eps>0$, one has
	\begin{equation*}
		\sup_{\lambda>0} \dE_{0,\lambda} [|\phi(L^{(\lambda)}) |^{1+\eps}] < \infty \, .
	\end{equation*}
	For $x\in [0,1]$, $|\phi(x)|\leq 1/e$. For $x\in [1,+\infty)$ and any $\eps \in (0,1)$, one has $\phi(x)^{1+\eps}\leq C_\eps x^2$ for some finite constant $C_\eps$. Thus
	\begin{equation*}
		\dE_{0,\lambda} [|\phi(L^{(\lambda)})|^{1+\eps}] \leq \frac{1}{e}+ C_\eps c \, ,
	\end{equation*}
	where $c$ is the constant appearing in \eqref{eq:proposition:KL_convergence:second_moment_constant}. The desired uniform integrability therefore holds, and point $(ii)$ follows. 
\end{proof}

{We can now apply Proposition \ref{proposition:KL_convergence} in our context:}
\begin{proposition}\label{prop:gaussian_KL_limit}
Denoting
\begin{equation}
	\KLls_d := \KL(\dPls_d || \dPl_d) \, ,
\end{equation} one has the following:
\begin{equation}\label{eq:prop:gaussian_KL_limit} 
	\forall d \geq 1, \; \KLls_d \underset{\lambda \to \infty}{\longrightarrow} \KLs_d = - \frac{1}{2}\sum_{\beta \in \cX_{d-1}} \log(1-s^{2 |\beta|}) = \frac{1}{2} \log\Cs_{d,2} \, 
\end{equation} 
where 
\begin{equation}\label{eq:rappel:KLs_bis}
	\Cs_{d,2} = \dEls_d[L_d] = \sum_{n \geq 1} A_{d,n}(s^2)^{n-1}  \, .
\end{equation}
\end{proposition}

\begin{proof}[Proof of Proposition \ref{prop:gaussian_KL_limit}]
The proof is the same as for Proposition \ref{prop:gaussian_KL_ineq}, except that we want in addition to apply Proposition \ref{proposition:KL_convergence},
with $\cX = \dR^{\cX_d} \times \dR^{\cX_d} $, $p_{0,\lambda}=\gwl_{d+1} \otimes \gwl_{d+1} $ and $p_{1,\lambda}=\dpls_{d+1} $. The weak convergence of $p_{0,\lambda}$ and $p_{1,\lambda}$ in the limit $\lambda \to \infty$ are ensured by 
Lemma \ref{lemma:cv_weak_dual},
hence the only additional thing that we need to check is that the likelihood ratio $\ell_\lambda$ of the pairs $(y,y')$ has a second moment that converges to that of the limiting likelihood ratio as $\lambda\to\infty$. Since the transformations $N \to y$ in \eqref{eq:y_alpha} and $N' \to y'$ in \eqref{eq:y'_alpha} are bijective, one has
\begin{equation*}
	\dE_{\gwl_{d} \otimes \gwl_{d}}[\ell_\lambda^2] = \dEl_{d}[L_d^2] = \dEls_{d}[L_d] \, ,
\end{equation*} which by \eqref{eq:rappel:KLs} does not depend on $\lambda$, and is readily seen to coincide with the second moment of the likelihood ratio between the limiting Gaussian distributions. Proposition \ref{proposition:KL_convergence} hence applies and we can conclude that 
\begin{equation*}
	\KLls_d \underset{\lambda \to \infty}{\longrightarrow} \KLs_d \, .
\end{equation*}
\end{proof}

\section{Other postponed proofs}

\subsection{Proof of Lemma \ref{lemma:initialize_lower_bound_on_KL}}
\label{appendix:proof:lemma:initialize_lower_bound_on_KL}

\begin{proof}[Proof of Lemma \ref{lemma:initialize_lower_bound_on_KL}]
	Since $s> \sqrt{\alpha}$, we have that $\KLs_{d} \to \infty$ when $d \to \infty$, in view of \eqref{eq:rappel:KLs}, the fact that $A_{d,n} \to A_n$ when $d \to \infty$, and Otter's formula \eqref{eq:theorem:otter}. For an arbitrarily large $K = K(c,\eps)$ to be specified later, we can thus choose $d_0 = d_0(s,c,\eps)$ such that $\KLs_{d_0}\geq K$. 
	
	Moreover, in view of \eqref{eq:prop:gaussian_KL_limitinf} 
	we can choose $\lambda_1 = \lambda_1(s,c,\eps)$ such that
	$$
	\lambda\geq \lambda_1\Rightarrow \KLls_{d_0}\geq \frac{1}{2}\KLs_{d_0}
	= \frac{1}{4}\log(\Cs_{d_0,2}) \, .
	$$
	We write then
	\begin{equation*}
		\KLls_{d_0}
		= \int_0^{\infty}\log(x)\dPls_{d_0}(L_{d_0}\in dx)
		\leq \int_1^{\infty}\log(x)\dPls_{d_0}(L_{d_0}\in dx)= \int_{1}^{\infty}\frac{1}{u}\dPls_{d_0}(L_{d_0}\geq u)du.
	\end{equation*}
	We consider two reals $A,B$ with $1<A<B$, and decompose the last integral as
	\begin{flalign*}
		& \int_{1}^{A}\frac{1}{u}\dPls_{d_0}(L_{d_0}\geq u)du
		+ \int_{A}^{B}\frac{1}{u}\dPls_{d_0}(L_{d_0}\geq u)du
		+ \int_{B}^{\infty}\frac{1}{u}\dPls_{d_0}(L_{d_0}\geq u)du \\
		\leq &
		\int_1^{A}\frac{1}{u}du+\dPls_{d_0}(L_{d_0}\geq A)\int_A^B\frac{du}{u}+\frac{1}{B} \int_B^\infty\dPls_{d_0}(L_{d_0}\geq u)du\\
		\leq & \log(A)+\dPls_{d_0}(L_{d_0}\geq A)\log(B/A)+\frac{1}{B}\Cs_{d_0,2} \, ,
	\end{flalign*} 
	where in the last step we have used
	\begin{equation*}
		\Cs_{d_0,2}=\dEls_{d_0}[L_{d_0}] = \int_0^{\infty} x\dPls_{d_0}(L_{d_0}\in dx) = \int_0^{\infty} \dPls_{d_0}(L_{d_0}\geq u)du \geq \int_B^{\infty} \dPls_{d_0}(L_{d_0}\geq u)du \, .
	\end{equation*}
	Combining these inequalities yields
	\begin{equation*}
		\frac{1}{4}\log(\Cs_{d_0,2}) \leq \log(A)+\dPls_{d_0}(L_{d_0}\geq A)\log(B/A)+\frac{1}{B}\Cs_{d_0,2} \, ,
	\end{equation*}
	hence
	\begin{equation*}
		\dPls_{d_0}(L_{d_0}\geq A)\geq \frac{\frac{1}{4}\log(\Cs_{d_0,2})-\log(A) -\frac{1}{B}\Cs_{d_0,2}}{\log(B/A)} \, .
	\end{equation*}
	Choosing $A=\left(\Cs_{d_0,2}\right)^{1/16}$ and $B=16 \Cs_{d_0,2}/\log(\Cs_{d_0,2})$ gives, after re-expressing the result in terms of $\KLs_{d_0} = \frac{1}{2} \log \Cs_{d_0,2}$, 
	\begin{equation*}
		\dPls_{d_0}(L_{d_0}\geq A)\geq 
		\frac{1}{4} \frac{\KLs_{d_0}}{\log(16)+\frac{15}{8}\KLs_{d_0} -\log(2 \KLs_{d_0}) } \, .
	\end{equation*}
	Since the function $K \mapsto \frac{1}{4} \frac{K}{\log(16)+\frac{15}{8}K -\log(2 K) } $ tends to $2/15$ as $K\to\infty$, for any $c\in(0,2/15)$ we can find a constant $K_1(c)$ such that $\KLs_{d_0} \geq K_1(c)$ implies $\dPls_{d_0}( L_{d_0}\geq A)\geq c$.
	
	Moreover, recalling that $\dEl_{d_0}[L_{d_0}]=1$ by construction of the likelihood ratio, Markov's inequality yields $\dPl_{d_0}(L_{d_0}\geq A) \leq A^{-1}$. Defining $K_2(\eps)=8 \log(1/\eps)$, one sees that $\KLs_{d_0} \geq K_2(\eps)$ implies $\dPl_{d_0}(L_{d_0}\geq A)\leq \eps$.
	
	The choice $K(c,\eps)=\max(K_1(c),K_2(\eps))$ thus ensures the statement of the Lemma, for the event $S=\set{L_{d_0}\geq A}$.
\end{proof}

\subsection{Proof of Proposition \ref{prop:bounding_LR_by_induction}} 
\label{appendix:proof:prop:bounding_LR_by_induction}

The proof of Proposition \ref{prop:bounding_LR_by_induction} relies on the following lemma. 
\begin{lemma}\label{lemma_bounds_Z}
	Assume $\lambda\geq 1$. The random variable $Z := Z_S$ defined in \eqref{eq:def_Z} verifies the following:
	\begin{itemize}
		\item[$(i)$] $\dEl_{d+1}[Z] = 0 \, .$
		\item[$(ii)$] $\dEls_{d+1}[Z] = \lambda s \dPls_{d}(S)\, .$
		\item[$(iii)$] $\dEl_{d+1}[Z^4] \leq 36 \lambda^4\dPl_d(S)^2+ 13\lambda^3 \dPl_d(S)\, .$
		\item[$(iv)$] $\Varls_{d+1}[Z] \leq \dEls_{d+1}[Z]+\lambda^2(1+s^2)\dPl_d(S)\, .$
	\end{itemize}
\end{lemma}
\begin{proof}[Proof of Lemma \ref{lemma_bounds_Z}]
	We shall rely on the moments of Poisson random variables, in particular the following identities for $X \sim \Poi(\mu)$, that follow from elementary computations: 
	\begin{equation}
		\label{eq:moments_Poisson}
		\dE[X]=\mu \ , \qquad \Var[X]=\mu \ , \qquad \dE[(X-\mu)^4]= 3 \mu^2 +\mu \, .
	\end{equation}
	
	Recall the definition of $Z=Z_S$:
	\begin{equation*}
		Z_S :=\sum_{(\tau,\tau')\in S}\widetilde{N}_\tau \widetilde{N}'_{\tau'} \, ,
	\end{equation*} where
	\begin{equation*}
		\widetilde{N}_\tau=N_\tau -\lambda \GWl_d(\tau) \quad \mbox{and} \quad \widetilde{N}'_\tau= N'_\tau-\lambda \GWl_d(\tau)\, .
	\end{equation*}
 
	\underline{\textit{Point $(i)$}} is immediate because under $\dPl_{d+1}$, for each pair $(\tau,\tau')\in\cX_d^2$,  the random variables $\widetilde{N}_\tau$, $\widetilde{N}'_{\tau'}$ are independent and have zero mean.\\
	
	\underline{\textit{Point $(ii)$}}: recall that under $\dPls_{d+1}$, $N$ and $N'$ are sampled as follows:
	\begin{equation*}
		N_{\tau} = \Delta_\tau + \sum_{\theta' \in \cX_{d}} M_{\tau, \theta'} \quad \mbox{and} \quad
		N'_{\tau} = \Delta'_{\tau} + \sum_{\theta \in \cX_{d}} M_{\theta, \tau'} \, ,
	\end{equation*} with $\Delta_{\tau}$, $\Delta'_{\tau}$ and $M_{\theta, \theta'}$ independent Poisson random variables with parameters $\lambda (1-s) \GWl_{d}(\tau)$ for the first two and $\lambda s \dPls_{d}(\theta, \theta')$ for the last one. We introduce the notations:
	\begin{equation*}
		\widetilde{\Delta}_\tau:=\Delta_\tau-\lambda (1-s) \GWl_d(\tau), \quad  \widetilde{\Delta}'_{\tau'}:=\Delta'_{\tau'}-\lambda (1-s) \GWl_d(\tau'), \quad \widetilde{M}_{\theta,\theta'}=M_{\theta,\theta'}-\lambda s \dPls_d(\theta,\theta') \, .
	\end{equation*}
	Since the marginals of $\dPls_d$ are given by $\GWl_d$, it holds that 
	\begin{equation*}
		\widetilde{N}_\tau=\widetilde{\Delta}_\tau+\sum_{\theta' \in \cX_{d}}\widetilde{M}_{\tau,\theta'} \quad \mbox{and} \quad \widetilde{N}'_{\tau'}=\widetilde{\Delta}'_{\tau'}+\sum_{\theta \in \cX_{d}} \widetilde{M}_{\theta,\tau'} \, ,
	\end{equation*} which shows that 
	\begin{equation*}
		\dEls_{d+1}[ \widetilde{N}_\tau \widetilde{N}'_{\tau'}]=\Varls_{d+1}( M_{\tau,\tau'})=\lambda s \dPls_d(\tau,\tau') \,.
	\end{equation*} Point $(ii)$  follows.\\
	
	\underline{\textit{Point $(iii)$}}: Write
	\begin{equation*}
		\dEl_{d+1}[Z^4] = \sum_{\substack{(\tau_1,\tau'_1) \in S, (\tau_2,\tau'_2) \in S \\ (\tau_3,\tau'_3) \in S (\tau_4,\tau'_4) \in S}} \dEl_{d+1}\left[ \prod_{i =1}^{4} \widetilde{N}_{\tau_i}\right] \dEl_{d+1}\left[ \prod_{i =1}^{4} \widetilde{N}'_{\tau'_i}\right] \, .
	\end{equation*} 
	Since the $\widetilde{N}_\tau$ are centered and independent the first expectation value in the above equation vanishes unless either $\tau_1=\tau_2=\tau_3=\tau_4$ or $\tau_1=\tau_2 \neq \tau_3=\tau_4$ (and two cases obtained by permutation of the indices). We let $\sigma(u,v)$ denote the summation of such terms with $|\set{\tau_i,i\in[4]}|=u$, $|\set{\tau'_i,i\in[4]}|=v$, for $u,v\in \set{1,2}$.
	
	We have $\sigma(1,1)=\sum_{(\tau,\tau')\in S}\dEl_{d+1}[\widetilde{N}^4_{\tau}] \dEl_{d+1} [\widetilde{N}'^4_{\tau'}]$. Using the value of the fourth centered moment of a Poisson random variable recalled in \eqref{eq:moments_Poisson}, and 
	the fact that $\GWl_d(\tau),\GWl_d(\tau')\leq 1$, one obtains
	\begin{flalign*}
		\sigma(1,1) & = \sum_{(\tau,\tau')\in S}\left[3\lambda^2 \GWl_d(\tau)^2+\lambda \GWl_d(\tau)\right]\left[3\lambda^2 \GWl_d(\tau')^2+\lambda \GWl_d(\tau')\right]\\
		&\leq 9\lambda^4\sum_{(\tau,\tau')\in S} \GWl_d(\tau)^2 \GWl_d(\tau')^2+[6\lambda^3+\lambda^2]\dPl_{d}(S) 
		\leq 9\lambda^4 \dPl_{d}(S)^2+ 7\lambda^3 \dPl_d(S) \,
	\end{flalign*} since $\lambda \geq 1$ and thanks to the easy bound $\sum_i x^2_i \leq \left( \sum_i x_i \right)^2$ for positive $x_i$.
	
	The term $\sigma(1,2)$ verifies
	\begin{flalign*}
		\sigma(1,2) & \leq  3\sum_{\tau} \dEl_{d+1}[\widetilde{N}^4_\tau]\sum_{\substack{\tau':(\tau,\tau')\in S \\ \theta':(\tau,\theta')\in S}}\dEl_{d+1}[\widetilde{N}'^2_{\tau'}]\dEl_{d+1}[\widetilde{N}'^2_{\theta'}]\\
		&= 3\sum_\tau \left[3\lambda^2 \GWl_d(\tau)^2+\lambda \GWl_d(\tau)\right] \sum_{\substack{\tau':(\tau,\tau')\in S \\ \theta':(\tau,\theta')\in S}}\lambda^2 \GWl_d(\tau') \GWl_d(\theta')\\
		&\leq   9 \lambda^4\sum_\tau \GWl_d(\tau)^2\sum_{\substack{\tau':(\tau,\tau')\in S \\ \theta':(\tau,\theta')\in S}}\GWl_d(\tau')\GWl_d(\theta')+3\lambda^3\sum_{(\tau,\tau')\in S}\GWl_d(\tau)\GWl_d(\tau') \, ,
	\end{flalign*} where we used the fact that $\sum_{\theta':(\tau,\theta')\in S}\GWl_d(\theta')\leq 1$. Note now that 
	\begin{equation*}
		\sum_{\substack{\tau':(\tau,\tau')\in S \\ \theta':(\tau,\theta')\in S}}\GWl_d(\tau')\GWl_d(\theta')\leq \dPl_d(S)^2 \qquad \mbox{and} \qquad \sum_\tau \GWl_d(\tau)^2 \leq 1
	\end{equation*}to conclude that $\sigma(1,2)\leq 9\lambda^4 \dPl_d(S)^2+3\lambda^3 \dPl_d(S)$; the same bound obviously holds also for $\sigma(2,1)$.\\
	
	Finally, $\sigma(2,2)$ can be bounded as follows. Having fixed $\tau_1$, there must be exactly one index $j\in\{2,3,4\}$ such that $\tau_j=\tau_1$. Consider thus that $j=3$ and $\tau_4=\tau_2$. By symmetry, when accounting only for this case, we just need to multiply our evaluation by 3. This leads to the following bound:
	\begin{flalign*}
		\sigma(2,2)&\leq 3\sum_{\tau_1,\tau_2}\lambda^2\GWl_d(\tau_1)\GWl_d(\tau_2)\sum_{\tau'_i,i\in[4]}\one_{(\tau_1,\tau'_1)\in S,(\tau_2,\tau'_2)\in S,(\tau_1,\tau'_3)\in S,(\tau_2,\tau'_4)\in S}\one_{|\{\tau'_i\}_i|=2}\, \dEl_d\left[\prod_{i=1}^{4} \widetilde{N}'_{\tau'_i}\right]\\
		&\leq 3\sum_{\tau_1,\tau_2}\lambda^2\GWl_d(\tau_1)\GWl_d(\tau_2)\sum_{\tau'_1,\tau'_2}3\lambda^2\GWl_d(\tau'_1)\GWl_d(\tau'_2)\one_{(\tau_1,\tau'_1)\in S,(\tau_2,\tau'_2)\in S} \, .
	\end{flalign*} Indeed, there are three possibilities for the choice of index $j'$ such that $\tau'_j=\tau'_1$, and for each such choice the contribution is upper bounded by the same term. This yields $\sigma(2,2)\leq 9\lambda^4 \dPl_d(S)^2$.
	
	Summing our bounds on $\sigma(u,v)$ for $u,v\in\set{1,2}$ yields $(iii)$. \\
	
	\underline{\textit{Point $(iv)$}}: Write
	\begin{multline*}
		\dEls_{d+1}(Z^2) \\
		= \dEls_{d+1} \sum_{(\tau,\tau')\in S}\sum_{(\theta,\theta')\in S}\left[\widetilde{\Delta}_\tau+\sum_{u'}\widetilde{M}_{\tau,u'}\right]\left[\widetilde{\Delta}'_{\tau'}+\sum_u \widetilde{M}_{u,\tau'}\right]\left[\widetilde{\Delta}_\theta+\sum_{v'}\widetilde{M}_{\theta,v'}\right]\left[\widetilde{\Delta}'_{\theta'}+\sum_v \widetilde{M}_{v,\theta'}\right] \, .
	\end{multline*} When expanding the product of brackets, the only terms that will yield a non-zero expectation must have the following sequence of degrees in variables $(\widetilde{\Delta},\widetilde{\Delta}',\widetilde{M})$: $(2,2,0)$, $(2,0,2)$, $(0,2,2)$, or $(0,0,4)$. Denote $\sigma(u,v,w)$ the summation of terms corresponding to exponents $(u,v,w)$. We have:
	\begin{flalign*}
		\sigma(2,2,0) =\sum_{(\tau,\tau')\in S}\dEls_d [\widetilde{\Delta}^2_\tau\widetilde{\Delta}'^2_{\tau'}] = \lambda^2(1-s)^2\dPl_d(S) \, .
	\end{flalign*} We next have
	\begin{flalign*}
		\sigma(2,0,2) = \sum_{(\tau,\tau')\in S}\dEls_d \left[\widetilde{\Delta}^2_\tau \sum_u\widetilde{M}^2_{u,\tau'}\right] 
		= \lambda^2 s (1-s)\dPl_d(S)\, ,
	\end{flalign*} and the same expression holds for $\sigma(0,2,2)$. We finally evaluate $\sigma(0,0,4)$. It reads
	\begin{flalign*}
		\sigma(0,0,4)=\sum_{\substack{(\tau,\tau')\in S \\ (\theta,\theta')\in S}}\sum_{u,u',v,v'}\dEls_{d+1} \left[ \widetilde{M}_{\tau,u'}\widetilde{M}_{u,\tau'}\widetilde{M}_{\theta, v'}\widetilde{M}_{v,\theta'} \right] \, .
	\end{flalign*} The non-zero terms in this expectation must comprise either the same term at the power 4, or two distinct terms each at power 2. This yields 4 contributions, that we denote by $A,B,C,D$, which satisfy
	\begin{flalign*}
		A & = \sum_{(\tau,\tau')\in S}\dEls_{d+1} [\widetilde{M}_{\tau,\tau'}^4] =  \sum_{(\tau,\tau')\in S}\left[ 3\lambda^2 s^2 \dPls_d(\tau,\tau')^2+ \lambda s \dPls_d(\tau,\tau')\right] \, , \\
		B & = \sum_{(\tau,\tau')\in S}\dEls_{d+1}[\widetilde{M}^2_{\tau,\tau'}]\sum_{\substack{(\theta,\theta')\in S \\ (\theta,\theta') \neq (\tau,\tau')}}\dEls_{d+1} [\widetilde{M}^2_{\theta,\theta'}] = \lambda^2 s^2 \dPls_d(S)^2-\lambda^2 s^2 \sum_{(\tau,\tau')\in S}\dPls_d(\tau,\tau')^2 \, ,
	\end{flalign*} and
	\begin{flalign*}
		C & = \sum_{(\tau,\tau')\in S} \left[\sum_{u'}\dEls_{d+1} [\widetilde{M}^2_{\tau,u'}]\right] \left[\sum_{u: (u,\tau')\ne (\tau,u')} \dEls_{d+1}[\widetilde{M}^2_{u,\tau'}] \right] \\
		& =
		\lambda^2 s^2 \dPl_d(S) -\lambda^2 s^2 \sum_{(\tau,\tau')\in S}\dPls_d(\tau,\tau')^2 \, , \\
		D & =  \sum_{(\tau,\tau')\in S} \sum_{\substack{(\theta,\theta')\in S \\ (\tau,\theta') \neq (\theta,\tau')}} \dEls_{d+1} [\widetilde{M}^2_{\tau,\theta'}] \dEls_{d+1} [\widetilde{M}^2_{\theta,\tau'}] \\
		& \leq 
		\sum_{(\tau,\tau')\in S}\sum_{\theta,\theta'}\dEls_{d+1} [\widetilde{M}^2_{\tau,\theta'}] \dEls_{d+1} [\widetilde{M}^2_{\theta,\tau'}] -\lambda^2 s^2\sum_{(\tau,\tau')\in S}\dPls_d(\tau,\tau')^2
		\\
		& =
		\lambda^2 s^2 \dPl_d(S)-\lambda^2 s^2\sum_{(\tau,\tau')\in S}\dPls_d(\tau,\tau')^2.
	\end{flalign*}
	
	Summing the expressions of $\sigma(2,2,0)$, $\sigma(2,0,2)$, $\sigma(0,2,2)$, $A$, $B$, $C$ and the upper bound of $D$ we obtain
	\begin{flalign*}
		\dEls_{d+1}[Z^2] & \leq  \dEls_{d+1}[Z]^2+\dEls_{d+1}[Z]+\lambda^2(1+s^2)\dPl_d(S) \, ,
	\end{flalign*} and upper bound $(iv)$ follows. 
\end{proof}

We now turn to the proof of Proposition \ref{prop:bounding_LR_by_induction}.
\begin{proof}[Proof of Proposition \ref{prop:bounding_LR_by_induction}] 
	Assuming that $S\subset \cX_d^2$ is such that
	\begin{equation*}
		\dPl_d(S)\leq \eps \quad \mbox{and} \quad \dPls_d(S)\geq c \, ,
	\end{equation*} our goal is to choose a threshold $\sigma \geq 0$ such that
	\begin{equation*}
		\dPl_{d+1}(Z\geq \sigma)\leq \frac{1}{2}\dPl_d(S) \leq \frac{\eps}{2}  \quad \mbox{and} \quad \dPls_{d+1}(Z\geq \sigma)\geq c \, .
	\end{equation*}
	
	\underline{\textit{First point.}} Using point $(iii)$ of Lemma \ref{lemma_bounds_Z}, and Markov's inequality we have
	\begin{equation*}
		\dPl_{d+1}(Z\geq \sigma) \leq \frac{1}{\sigma^4}\dEl_{d+1} [Z^4]\leq \frac{1}{\sigma^4} (36 \lambda^4\dPl_{d}(S)^2+13\lambda^3 \dPl_{d}(S)) \, .
	\end{equation*} It thus suffices to choose $\sigma^4=\max\left(144\lambda^4 \dPl_d(S),52 \lambda^3\right)$ to ensure the first property, that is guarantying that $\dPl_{d+1}(Z\geq \sigma)\leq \frac{1}{2}\dPl_d(S)$. We can a fortiori take $\sigma=\max(4 \lambda \dPl_d(S)^{1/4}, 3 \lambda ^{3/4})$.
	
	\underline{\textit{Second point.}} By point $(ii)$ of Lemma \ref{lemma_bounds_Z}, since $\dEls_{d+1}[Z]=\lambda s \dPls_d(S)\geq \lambda s c$  we shall have $\dEls_{d+1}[Z]\geq 2\sigma$ provided
	$$
	8 \dPl_d(S)^{1/4}\leq s c \hbox{ and } 6 \lambda ^{-1/4} \leq s c \, ,
	$$ or equivalently
	\begin{equation}\label{eq:tmp_epsilon_0_lambda_0}
		\dPl_d(S)\leq \left(\frac{sc}{8}\right)^4\hbox{ and }\lambda \geq \left(\frac{6}{sc}\right)^4.
	\end{equation}
	This provides the conditions on $\lambda_0$ and $\eps$ required in the statement of the proposition, but let us assume that \eqref{eq:tmp_epsilon_0_lambda_0} is satisfied for now. Using $\sigma\leq \dEls_{d+1}[Z]/2$, Chebyshev's inequality as well as the bound $(iv)$ of Lemma \ref{lemma_bounds_Z}:
	\begin{flalign*}
		\dPls_{d+1}(Z\leq \sigma)& \leq \dPls_{d+1}\left( |Z-\dEls_{d+1}[Z]|\geq \frac{1}{2}\dEls_{d+1}[Z]\right)
		\leq 4 \frac{\Varls_{d+1}(Z)}{\dEls_{d+1}[Z]^2} \\
		&\leq 4\frac{ \lambda s \dPls_{d}(S)+2 \lambda^2 \dPl_d(S)}{\lambda^2 s^2 \dPls_{d}(S)^2}
		\leq \frac{4}{\lambda s c}+\frac{8 \dPl_d(S)}{s^2 c^2} \, .
	\end{flalign*}
	In order to ensure that $\dPls_{d+1}(Z\leq \sigma)\leq 1-c$, it thus suffices to ensure
	\begin{equation*}
		\frac{4}{\lambda s c}+\frac{8\dPl_d(S)}{s^2 c^2}\leq 1-c \, .
	\end{equation*}
	We can for instance impose
	\begin{equation*}
		\lambda\geq \frac{8}{s c (1-c)}, \quad \dPl_d(S)\leq \frac{(1-c)s^2 c^2}{16}\, .
	\end{equation*}
	Combining this requirement with \eqref{eq:tmp_epsilon_0_lambda_0} we have the announced property by setting
	$$
	\lambda\geq \lambda_0(s,c):=\max\left(\frac{8}{s c (1-c)},\left(\frac{6}{sc}\right)^4\right),\; \dPl_d(S)\leq \eps(s,c):=\min\left( \left(\frac{sc}{8}\right)^4, \frac{(1-c)s^2 c^2}{16}\right).
	$$
	Since $\dPls_{d+1}(Z\geq \sigma) \geq \dPls_{d+1}(Z > \sigma) = 1 - \dPls_{d+1}(Z \leq \sigma) $ the statement follows.
\end{proof}

\end{appendix}

\end{document}